\RequirePackage{etoolbox}
\csdef{input@path}{{style/}{graphics/}}
\documentclass[MSNbibl,number,citesort,dvips]{arxbj}

\aid{0}
\volume{21}
\issue{1}
\pubyear{2015}
\firstpage{537}
\lastpage{573}
\doi{10.3150/13-BEJ579} 

\makeatletter

\newcommand{\rright}{\right}
\newcommand{\lleft}{\left}
\newcommand{\rrvert}{\vert}
\newcommand{\llvert}{\vert}

\newproclaim{defi}{Definition}
\newtheorem{prop}{Proposition}
\newtheorem{cor}{Corollary}
\newproclaim{assumption}{Assumption}
\newremark{ex}{Example}
\newtheorem{lemma}{Lemma}
\newremark{remark}{Remark}
\newremark{LBM}{Latent block model (LBM)}
\newremark{SBM}{Stochastic block model (SBM)}

\newcommand{\pr}{\mathbb{P}}
\newcommand{\esp}{\mathbb{E}}
\newcommand{\ba}{\boldsymbol{\alpha}}
\newcommand{\bb}{\boldsymbol{\beta}}
\newcommand{\bmu}{\boldsymbol{\mu}}
\newcommand{\bpi}{\boldsymbol{\pi}}

\newcommand{\Sig}{\mathfrak{S}}

\newcommand{\X}{\mathcal{X}}
\newcommand{\I}{\mathcal{I}}
\newcommand{\Q}{\mathcal{Q}}
\newcommand{\calL}{\mathcal{L}}
\newcommand{\N}{\mathbb{N}}
\newcommand{\U}{\mathcal{U}}
\newcommand{\tU}{\tilde{\mathcal U}}

\newcommand{\eps}{\varepsilon}

\newcommand{\argmax}{\operatorname{\arg\max}}
\newcommand{\eqref}[1]{(\ref{#1})}
\newcommand{\diff}{\operatorname{diff}}
\renewcommand{\smallsetminus}{\setminus}
\newcommand{\bin}{\mathrm{bin}}
\makeatother

\begin{document}
\begin{frontmatter}

\title{Convergence of the groups posterior distribution in latent or
stochastic block~models}
\runtitle{Groups posterior in LBMs and SBMs}

\begin{aug}
\author[1]{\inits{M.}\fnms{Mahendra} \snm{Mariadassou}\thanksref{1}\ead[label=e1]{mahendra.mariadassou@jouy.inra.fr}} 
\and
\author[2]{\inits{C.}\fnms{Catherine} \snm{Matias}\thanksref{2}\ead[label=e2]{catherine.matias@genopole.cnrs.fr}}
\address[1]{INRA, UR1077 Unit\'{e} Math\'{e}matique, Informatique et G\'{e}nome, 78350
Jouy-en-Josas, France.\\ \printead{e1}}
\address[2]{Laboratoire Statistique et G\'{e}nome, Universit\'{e} d'\'
{E}vry Val d'Essonne,
UMR CNRS 8071 -- USC INRA, 23 bvd de France, 91 037 \'{E}vry, France.
\printead{e2}}
\end{aug}

\received{\smonth{6} \syear{2012}}
\revised{\smonth{11} \syear{2013}}

%
\begin{abstract}
We propose a unified framework for studying both latent and
stochastic block models, which are used to cluster simultaneously
rows and columns of a data matrix. In this new framework, we study
the behaviour of the groups posterior distribution, given the data.
We characterize whether it is possible to asymptotically recover the
actual groups on the rows and columns of the matrix, relying on a
consistent estimate of the parameter. In other words, we establish
sufficient conditions for the groups posterior distribution to
converge (as the size of the data increases) to a Dirac mass located
at the actual (random) groups configuration. In particular, we
highlight some cases where the model assumes symmetries in the
matrix of connection probabilities that prevents recovering the
original groups. We also discuss the validity of these results when
the proportion of non-null entries in the data matrix converges to
zero.
\end{abstract}

%
\begin{keyword}
\kwd{biclustering}
\kwd{block clustering}
\kwd{block modelling}
\kwd{co-clustering}
\kwd{latent block model}
\kwd{posterior distribution}
\kwd{stochastic block model}
\end{keyword}

\end{frontmatter}

\section{Introduction}\label{intro}

Cluster analysis is an important tool in a variety of scientific areas
including pattern recognition, microarrays analysis, document
classification and more generally data mining. In these contexts, one
is interested in data recorded in a table or matrix, where for
instance rows index objects and columns index features or variables.
While the majority of clustering procedures aim at clustering either
the objects or the variables, we focus here on procedures which
consider the two sets simultaneously and organize the data into
homogeneous blocks. More precisely, we are interested in
probabilistic models called latent block models (LBMs), where both
rows and columns are partitioned into latent groups
\cite{Govaert_Nadif_PatternRec03}.

Stochastic block models (SBMs, \cite{Holland_etal_83}) may be viewed
as a particular case of LBMs where data consists in a random graph
which is encoded in its adjacency matrix. An adjacency matrix is a
square matrix where rows and columns are indexed by the same set of
objects and an entry in the matrix describes the relation between two
objects. For instance, binary random graphs are described by a binary
matrix where entry $(i,j)$ equals 1 if and only if there is an edge
between nodes $(i,j)$ in the graph. Similarly, weighted random graphs
are encoded in square matrices where the entries describe the edges
weights (the weight being 0 in case of no edge between the two nodes).
In this context the partitions on rows and columns of the square
matrix are further constrained to be identical.

To our knowledge and despite their similarities, LBMs and SBMs have
never been explored from the same point of view. We aim at presenting
a unified framework for studying both LBMs and SBMs. We are more
precisely interested in the behaviour of the groups posterior
distribution, given the data. Our goal is to characterize whether it
is possible to asymptotically recover the actual groups on the rows
and columns of the matrix, relying on a consistent estimate of the
parameter. In other words, we establish sufficient conditions for the
groups posterior distribution to converge (as the size of the data
increases) to a Dirac mass located at the actual (random) groups
configuration. In particular, we highlight some cases where the model
assumes symmetries in the matrix of connection probabilities that
prevents recovering the original groups (see
Theorem~\ref{thm:cv_posterior} and following corollaries). Note that
the asymptotic framework is particularly suited in this context
as the datasets are often huge.

One of the first occurrences of LBMs appears in the pioneering work
\cite{Hartigan} under the name \emph{three partitions}. LBMs were
later developed as an intuitive extension of the finite mixture model,
to allow for simultaneous clustering of objects and features. Many
different names are used in the literature for such procedures, among
which we mention block clustering, block modelling, biclustering,
co-clustering and two-mode clustering. All of these procedures differ
through the type of clusters they consider. LBMs induce a specific
clustering on the data matrix, namely we \emph{partition} the rows and
columns of the data matrix and the data clusters are restricted to
\emph{cartesian products} of a row cluster and a column cluster.
Frequentist parameter estimation procedures for LBMs have been
proposed in \cite{Govaert_Nadif_PatternRec03,Govaert_Nadif_CSDA08} for
binary data and in \cite{Govaert_Nadif_CS10} for Poisson random
variables. A Bayesian version of the model has been introduced in
\cite{DeSarbo} for random variables belonging to the set $[0,1]$,
combined with a Markov chain Monte Carlo (MCMC) procedure to estimate
the model parameters. Moreover, model selection in a Bayesian setting
is performed at the same time as parameter estimation in
\cite{Wyse_Friel}, that considers two different types of models: a
Bernoulli LBM for binary data and a Gaussian one for continuous
observations. All of these parameter estimation procedures also
provide a clustering of the data, based on the groups posterior
distribution computed at the estimated parameter value. In the
following, \emph{a posteriori estimation} of the groups refers to
maximum a
posteriori (MAP) procedure on the groups posterior distribution
computed at some estimated parameter value.
To our knowledge,
there is no result in the literature about the quality of such
clustering procedures nor about convergence of the groups posterior
distribution in LBMs.

SBMs were (re)-discovered many different times in the literature, and
introduced at first in social sciences to study relational data
(see, for instance, \cite{Daudin,FH82,Holland_etal_83,SN97}). In this
context, the data consists in a random graph over a set of nodes, or
equivalently in a square matrix (the adjacency matrix) whose entries
characterize the relation between two nodes. The nodes are partitioned
into latent groups so that the clustering of the rows and columns of
the matrix is now constrained to be identical. Various parameter
estimation procedures have been proposed in this context, from
Bayesian strategies \cite{NS01,SN97}, to variational approximations
of expectation maximization (EM) algorithm
\cite{Daudin,Mariadassou_10,Picard_BMC} or variational Bayes
approaches \cite{latouche_Bayes_EM}, online procedures
\cite{Zanghi_PatternRec08,Zanghi_AAS10} and direct methods
\cite{Ambroise_Matias,Channarond}. Note that most of these works are concerned
with binary data and only some of the most recent of them deal with
weighted random graphs \cite{Ambroise_Matias,Mariadassou_10}.

In each of these procedures, a clustering of the graph nodes is
performed according to the groups posterior distribution (computed at the
estimated parameter value). The behaviour of this posterior
distribution for binary SBMs is studied in \cite{Alain_JJ_LP}. These
authors establish two different results. The first one (Theorem~3.1 in \cite{Alain_JJ_LP})
states that at the true parameter value, the
groups posterior distribution converges to a Dirac mass at the actual
value of groups configuration (controlling also the corresponding rate
of convergence). This result is valid only at the true parameter
value, while the above mentioned procedures rely on the groups
posterior distribution at an estimated value of the parameter instead
of the true one. Note also that this result establishes a convergence
under the \emph{conditional} distribution of the data, given the
actual configuration on the groups. However, as this convergence is
uniform with respect to the actual configuration, the result also
holds under the unconditional distribution of the observations. The
second result they obtain on the convergence of the groups posterior
distribution (Proposition~3.8 in \cite{Alain_JJ_LP}) is valid at an
estimated parameter value, provided this estimator converges at rate
at least $n^{-1}$ to the true value, where $n$ is the number of nodes
in the graph (number of rows and columns in the square data
matrix). Note that this latter assumption is not harmless as it is not
established that such an estimator exists, except in a particular
setting \cite{Ambroise_Matias}; see also \cite{Gazal} for empirical
results. There are thus many differences between our result
(Theorem~\ref{thm:cv_posterior} and following corollaries) and theirs:
we provide a result for any parameter value in the neighborhood of the
true value, we work with non-necessarily binary data and our work
encompasses both SBMs and LBMs. We however mention that the main goal
of these authors is different from ours and consists in establishing
the consistency of maximum likelihood and variational estimators in
SBMs.

We stress here that our result relies on the existence of
\emph{consistent} parameter estimates (without any constraint on
the convergence rate). Such consistency results
have been established for instance, in \cite{Ambroise_Matias} in the specific
context of affiliation (namely only two connections types are
considered: intra-group and inter-group connections) for binary or
weighted SBMs; in \cite{Alain_JJ_LP} for binary (possibly
directed) SBMs and concerning the connectivity parameters (a result on the
groups proportions requires an additional assumption whose validity
is not yet established) and also by \cite{Bickel_Chen} for binary
SBMs where the parameter estimates are derived from groups
estimators that rely on specific consistent modularities.
As already stressed in
the above paragraph, the behaviour of the groups posterior
distribution is not fully resolved in those contexts.
Moreover, to our knowledge, consistency results have not
been (theoretically) established in LBMs but we believe that our
common framework enables to obtain results in LBMs similar as those
obtained in SBMs.

Let us now discuss the articles \cite{Bickel_Chen,Choi_etal} and
\cite{Rohe_Chat} on
the performance of clustering procedures for random graphs, as well as
the very recent works \cite{Flynn_Perry,Rohe_Yu}
on the performance of co-clustering procedures. Those articles, which
are of a
different nature from ours, establish that under some conditions, the
fraction of misclassified nodes (resulting from different algorithmic
procedures) converges to zero as the number of nodes increases. These
results, with the exception of \cite{Flynn_Perry}, apply to binary
graphs only, while we shall deal both with binary and weighted graphs;
as well as real-valued array data. Moreover, they establish results
on procedures that estimate parameters while clustering the
data, while we are rather interested in MAP based procedures
relying on any consistent parameter estimate. In
\cite{Bickel_Chen}, Bickel and Chen show that groups estimates based
on the use of
different modularities are consistent in the sense that with
probability tending to one, these recover the original groups. Rohe
and Yu \cite{Rohe_Yu} are concerned with a framework in which nodes
of a
graph belong to two groups: a \emph{receiver} group and a
\emph{sender} group. This is a refinement of standard SBM, which
assumes equal sender and receiver groups, and is motivated by the
study of directed graphs. The results of \cite{Rohe_Yu} are very
similar to those of \cite{Rohe_Chat} that apply on symmetric binary
graphs: they propose a classification algorithm based on spectral
clustering that achieves vanishing classification error rate. Flynn
and Perry \cite{Flynn_Perry} share our framework with a few
exceptions: they
replace the profile likelihood (used for instance,
in \cite{Bickel_Chen}) with a rate function allowing for model
misspecification. Besides, they model sparsity with a scaling
parameter acting directly on the mean interaction value instead of
inflating the number of zero-valued interactions with a Bernoulli
variable as we do (see Section~\ref{sec:weighted}). They essentially
extend the results of
\cite{Choi_etal} to weighted graphs. The focus on the mean
interaction value allows for model misspecification but prevents the
detection of groups that differ mostly in interaction variance.
Indeed, the simulation study from \cite{Flynn_Perry} considers only
groups varying in their mean interaction value, while we can detect
groups varying through their variance for instance. We also mention
that \cite{Choi_etal,Rohe_Chat} and \cite{Rohe_Yu} are
concerned with an asymptotic setting where the number of groups is
allowed to grow with network size and the average network degree grows
at least nearly linearly \cite{Rohe_Chat,Rohe_Yu} or
poly-logarithmically \cite{Choi_etal} in this size. In
Section~\ref{sec:pi_vers0} of the present work, we explore the
validity of our results in a similar framework, by assuming that the
numbers of groups remain fixed while the connections probabilities
between groups converge to zero. Finally and most importantly, note
that \emph{all these works but} \cite{Bickel_Chen} propose
convergence results in a setup of \emph{independent} random
observations (and Bernoulli distributed, except for \cite
{Flynn_Perry} that consider more general distributions),
viewing the latent groups as parameters instead of random variables.
On the contrary in our context,
the observed random variables are non-independent. This makes a
tremendous difference in the validity of the statements.

We also want to outline that many different generalization allowing
for overlapping groups exist, both for LBMs and SBMs. We refer the
interested reader to the works \cite{DeSarbo} for LBMs and
\cite{Airoldi,Latouche_overlap} in the case of SBMs, as well as the
references therein. However in this work, we restrict our attention
to non-overlapping groups.

This work is organized as follows. Section~\ref{sec:model} describes
LBMs and SBMs and introduces some important concepts such as
equivalent group configurations. Section~\ref{sec:post-distr}
establishes general and sufficient conditions for the groups posterior
probability to converge (with large probability) to a (mixture of)
Dirac mass, located at (the set of configurations equivalent to) the
actual random configuration. In particular, we discuss the cases where
it is likely that groups estimation relying on maximum posterior
probabilities might not converge. Section~\ref{sec:applis}
illustrates our main result, providing a large number of examples
where the above mentioned conditions are satisfied. Finally, in
Section~\ref{sec:pi_vers0} we explore the validity of our results when
the connections probabilities between groups converge to zero. This
corresponds to datasets with an asymptotically decreasing density of
non-null entries. Some technical proofs are postponed to the
\hyperref[sec:appendix]{Appendix}.

\section{Model and notation}\label{sec:model}
\subsection{Model and assumptions}

We observe a matrix $\mathbf{X}_{n,m}:=\{X_{ij}\}_{1\le i\le n, 1\le
j\le m}$ of
random variables in some space set $\X$, whose distribution is
specified through latent groups on the rows and columns of the matrix.

Let $Q\ge1$ and $L\ge1$ denote the number of latent groups
respectively on the rows and columns of the matrix. Consider the
probability distributions $\ba=(\alpha_1,\ldots, \alpha_Q)$ on $\Q=
\{1,\ldots, Q\}$ and $\bb=(\beta_1,\ldots,\beta_L)$ on
$\calL=\{1,\ldots,L\}$, such that
\[
\forall q \in\Q, \forall l\in\calL,\qquad  \alpha_q, \beta_l>0
\quad \mbox{and}\quad  \sum_{q=1}^Q
\alpha_q=1,\qquad  \sum_{l=1}^L
\beta_l=1.
\]
Let $\mathbf{Z}_n:=Z_1,\ldots, Z_n$ be independent and identically
distributed (i.i.d.) random variables, with distribution $\ba$ on
$\Q$ and $\mathbf{W}_{m}:=W_1,\ldots, W_m$ i.i.d. random variables with
distribution $\bb$ on $\calL$. Two different cases will be considered
in this work:
\begin{LBM*}
 In this case, the random
variables $\{Z_i\}_{1\le i\le n}$ and $\{W_j\}_{1\le j\le m}$ are
independent. We let $\I=\{1,\ldots,n\}\times\{1,\ldots,m\}$ and
$\bmu=\ba^{\otimes n} \otimes\bb^{\otimes m}$ the distribution of
$(\mathbf{Z}_n, \mathbf{W}_{m}) :=(Z_1,\dots,Z_n,W_1,\dots,W_m)$
and set $U_{ij} =
(Z_i, W_j)$ for $(i,j)$ in $\I$. The random vector $(\mathbf
{Z}_n,\mathbf{W}_{m})$
takes values in the set $\U:=\Q^n\times\calL^m$ whereas the
$\{U_{ij}:=(Z_i,W_j)\}_{(i,j)\in\I}$ are non-independent random
variables taking values in the set $(\Q\times\calL)^{nm}$.
\end{LBM*}
\begin{SBM*} In this case, we have $n=m,
\Q=\calL$, $Z_i=W_i$ for all $1\le i \le n$ and $\ba=\bb$. We let
$\I=\{1,\ldots,n\}^2$, $\bmu= \ba^{\otimes n}$ the distribution of
$\mathbf{Z}_n$ and set $U_{ij} = (Z_i,Z_j)$ for $(i,j) \in\I$. The random
variables $\{U_{ij} :=(Z_i,Z_j)\}_{(i,j)\in\I}$ are
not independent and take values in the set
\[
\U=\bigl\{\bigl\{(q_i,q_j)\bigr\}_{ (i,j)\in\I} ;
\forall i\in \{1,\ldots,n\} , q_i\in\Q\bigr\}.
\]
This case corresponds to the observation of a random graph whose
adjacency matrix is given by $\{X_{ij}\}_{1\le i, j \le n }$. As
particular cases, we may also consider graphs with no self-loops in
which case $\I=\{1,\ldots,n\}^2\smallsetminus\{(i,i); 1\le i\le
n\}$. We may also consider undirected random graphs, possibly with
no self-loops, by imposing symmetric adjacency matrices
$X_{ij}=X_{ji}$. In this latter case, $\I=\{1\le i<j \leq n\}$.
\end{SBM*}

In the following, we refer to each of these two cases by indicating
the symbols (LBM) and (SBM). Whenever possible, we give general
formulas valid for the two cases, and which could be simplified
appropriately in SBM. We introduce a matrix of connectivity
parameters $\bpi=(\pi_{ql})_{(q,l)\in\Q\times\calL}$ belonging to
some set of matrices $\Pi_{\Q\calL}$ whose coordinates $\pi_{ql}$
belong to some set $\Pi$ (note that $\Pi_{\Q\calL}$ may be different
from the product set $\Pi^{QL}$). Now, conditional on the latent
variables $\{U_{ij}=(Z_i,W_j)\}_{(i,j)\in\I}$, the observed random
variables $\{X_{ij}\}_{(i,j)\in\I}$ are assumed to be independent,
with a parametric distribution on each entry depending on the
corresponding rows and columns groups. More precisely, conditional on
$Z_i=q$ and $W_j=l$, the random variable $X_{ij}$ follows a
distribution parameterized by $\pi_{ql}$. We let $f(\cdot;\pi_{ql})$
denote its density with respect to some underlying measure (either the
counting or Lebesgue measure).

The model may be summarized as follows:
%
\begin{equation}
\label{eq:model_gen} %
\begin{array} {l@{\quad}l} \bullet&(\mathbf{Z}_n,
\mathbf{W}_{m}) \mbox{ latent random variables in $\U$ with
distribution given by }\bmu,
\\
\bullet&\mathbf{X}_{n,m}= \{X_{ij}\}_{(i,j)\in\I} \mbox{
observations in } \X,
\\
\bullet& \pr(\mathbf{X}_{n,m}| \mathbf{Z}_n,
\mathbf{W}_{m}) =\bigotimes_{(i,j)\in\I} \pr
(X_{ij} | Z_{i}, W_j ),
\\
\bullet& \forall{(i,j)\in\I} \mbox{ and } \forall(q,l) \in \Q\times\calL, \mbox{
we have } X_{ij} | (Z_{i}, W_{j})=(q,l) \sim
f(\cdot; \pi_{ql}). \end{array} %
\end{equation}
We consider the following parameter set
\[
\Theta= \bigl\{ \theta= (\bmu,\bpi) ; \bpi\in\Pi_{\Q\calL} \mbox{ and }
\forall(q,l)\in\Q\times \calL , \alpha_q \ge\alpha_{\min} >
0 , \beta_l \ge\beta_{\min} >0 \bigr\} ,
\]
and define $\alpha_{\max} = \max\{\alpha_q ;q \in\Q; \theta\in
\Theta\}$ and similarly $\beta_{\max} = \max\{\beta_l ;l \in\calL;
\theta\in\Theta\}$. We let $\mu_{\min} :=\alpha_{\min}
\wedge
\beta_{\min}$ and $\mu_{\max} :=\alpha_{\max} \vee
\beta_{\max}$. Note that in SBM, $\mu_{\min} $ (resp. $\mu_{\max}$)
reduces to $\alpha_{\min}$ (resp. $\alpha_{\max}$). We denote by
$\mathbb{P}_{\theta}$ and $\mathbb{E}_{\theta}$ the probability
distribution and expectation under
parameter value $\theta$. In the following, we assume that the
observations $\mathbf{X}_{n,m}$ are drawn under the true parameter
value $\theta^{\star}
\in\Theta$. We let $\mathbb{P}_{\star}$ and $\mathbb{E}_{\star}$
respectively, denote
probability and expectation under parameter value $\theta^{\star}$.
We now
introduce a necessary condition for the connectivity parameters to be
identifiable from $\pr_\theta$.

\begin{assumption}\label{hyp:identifiability}
\begin{enumerate}[(ii)]
\item[(i)] The parameter $\pi\in\Pi$ is identifiable from the
distribution $f(\cdot;\pi)$, namely $f(\cdot;\pi)=f(\cdot;\pi')
\Rightarrow\pi=\pi'$.
\item[(ii)] For all $q \neq q' \in\Q$, there exists some $l \in
\calL$ such that $ \pi_{ql} \neq\pi_{q'l}$. Similarly, for all $l
\neq l' \in\calL$, there exists some $q \in\Q$ such that $
\pi_{ql} \neq\pi_{ql'}$.
\end{enumerate}
\end{assumption}

Assumption~\ref{hyp:identifiability} will be in force throughout this
work. Note that it is a very natural assumption. In particular, (i)
will be satisfied by any reasonable family of distributions and if
(ii) is not satisfied, there exist for instance, two row groups $q\neq
q'$ with the same behavior. These groups (and thus the corresponding
parameters) may then not be distinguished relying on the marginal
distribution of $\pr_\theta$ on the observation space $\X^\N$. Note
also that Assumption~\ref{hyp:identifiability} is in general not
sufficient to ensure identifiability of the parameters in LBM or SBM.
Identifiability results for SBM have first been given in a particular
case in \cite{ECJ} and then later more thoroughly discussed in
\cite{ident_mixnet} for undirected, binary or weighted random graphs.
See also \cite{Alain_JJ_LP} for the case of directed and binary random
graphs.

In the following, for any subset $A$ we denote by either $1_A$ or
$1\{A\}$ the indicator function of event $A$, by $|A|$ its cardinality
and by $\bar A$ the complementary subset (in the ambient set).

\subsection{Equivalent configurations}
First of all, it is important to note that the classical label
switching issue that arises in any latent variable model also takes
place in LBMs and SBMs. As such, any permutation on the labels of the
rows and columns groups will induce the same distribution on the data
matrix. To be more specific, we let $\Sig_Q$ (resp. $\Sig_L$) be the
set of permutations of $\Q$ (resp. $\calL$). In the following, we
define $\mathfrak{S}_{\Q\calL}$ to be either the set $\Sig_Q\times
\Sig_L$ (LBM) or
the set $\{(s,s) ; s \in\Sig_Q\}$ (SBM). We consider some $\sigma
\in\mathfrak{S}_{\Q\calL}$ and for any parameter value $\theta
=(\bmu,\bpi)$ we denote
by $\sigma(\theta)$ the parameter induced by permuting the labels of
rows and columns groups according to $\sigma$. Then label switching
corresponds to the fact that
%
\begin{equation}
\label{eq:label_switch} \mathbb{P}_{\theta}=\mathbb{P}_{\sigma(\theta)} .
\end{equation}
Now in LBM and SBM, there exists an additional phenomenon,
that is specific to these models and comes from the fact that the
distribution of any random variable depends on two different latent
ones. Let us explain this now. In a classical latent variable model
where the distribution of individuals belonging to group $q$ is
characterised by the parameter $\pi_q$, identifiability conditions
will require that $\pi_q\neq\pi_l$ for any two different groups
$q\neq l$. Now, when considering two groups characterising the
distribution of one random variable, it may happen that for instance
$\pi_{ql}=\pi_{q'l}$ for two different groups $q\neq q'$. Indeed, the
groups $q,q'$ may be differentiated through their connectivity to
other groups than the group $l$ (see point (ii) in
Assumption~\ref{hyp:identifiability}). As a consequence, if the
parameter matrix $\bpi$ has some symmetries (which is often the case
for model parsimony reasons), it may happen that some row and column
groups can be permuted while the connectivity matrix $\bpi$ remains
unchanged. Note that in this case, the global parameter
$\theta=(\bmu,\bpi)$ remains identifiable as soon as the groups
proportions (characterised by $\bmu$) are different. More precisely,
it may happen that for some $\sigma\in\mathfrak{S}_{\Q\calL}$, we
have $\bpi=\sigma
(\bpi)$ (a case that can
never occur for simple latent variables models) and thus
%
\begin{equation}
\label{eq:equiv_conf} \pr_{\bmu,\bpi}=\mathbb{P}_{\bmu,\sigma(\bpi)}.
\end{equation}
Note the difference between \eqref{eq:label_switch} and
\eqref{eq:equiv_conf}. In particular, whenever $\bmu\neq\sigma(\bmu)$
we have $\pr_{\bmu,\bpi} \neq\mathbb{P}_{\sigma(\bmu),\sigma
(\bpi)}$
and we are not facing an instance of label switching.

We now formalize the concept of \emph{equivalent configurations} that
will enable us to deal with possible symmetries in the parameter
matrices $\bpi$. Note that from a practical perspective, these
subtleties have little impact (in fact, the same kind of impact as the
label switching). But these are necessary for stating our results
rigorously.

For any $(s,t)\in\mathfrak{S}_{\Q\calL}$, we let
\begin{eqnarray*}
\bpi^{s,t} := \bigl( \pi^{s,t}_{ql} \bigr)
_{(q,l)\in\Q\times\calL} : = ( \pi_{s(q)t(l)}) _{(q,l)\in\Q\times\calL} .
\end{eqnarray*}
Fix a subgroup $\Sig$ of $\mathfrak{S}_{\Q\calL}$ and a parameter
set $ \Pi_{\Q
\calL}$. Whenever for any pair of permutations $(s,t) \in
\Sig$ and any parameter $\bpi\in\Pi_{\Q\calL}$ we have
$\bpi^{s,t}=\bpi$, we say that the parameter set $\Pi_{\Q\calL}$ is
\emph{invariant under the action of $\Sig$}. In the following, we
will consider parameter sets that are invariant under some subgroup
$\Sig$. This includes the case where $\Sig$
is reduced to identity. We will moreover exclude from the parameter
set $\Pi_{\Q\calL}$ any point $\bpi$ admitting specific symmetries,
namely such that there exists
\[
(s,t)\in\mathfrak{S}_{\Q\calL}\setminus\Sig\quad  \mbox{satisfying}\quad
\bpi^{s,t}=\bpi.
\]

Note that this corresponds to excluding a subset of null Lebesgue
measure from the parameter set $\Pi_{\Q\calL}$.

\begin{assumption}
\label{hyp:symetrie}
The parameter set $\Pi_{\Q\calL}$ is invariant under the action of
some (maximal) subgroup $\Sig$ of $\mathfrak{S}_{\Q\calL}$.
Moreover,
for any pair of permutations $(s,t)\in\mathfrak{S}_{\Q\calL}
\setminus\Sig$
and any parameter $\bpi\in\Pi_{\Q\calL}$, we
assume that $\bpi^{s,t}\neq\bpi$.
\end{assumption}

\begin{ex}
In SBM, we consider $\Sig=\{(\mathit{Id},\mathit{Id})\}$ where $\mathit{Id}$ is the identity
and let
\[
\Pi_{\Q\calL} =\bigl\{\bpi\in\Pi^{Q^2}; \forall s\in
\Sig_Q, s\neq \mathit{Id}, \mbox{ we have } \bpi^{s,s}\neq \bpi\bigr
\}.
\]
\end{ex}

\begin{ex}[(Affiliation SBM)] \label{ex:affiliation}
In SBM, we consider $\Sig=\{(s,s) ; s \in\Sig_Q\}$ and let
$\Pi_{\Q\calL} = \{ (\lambda-\nu) I_Q + \nu\mathbf
{1}_Q^\intercal
\mathbf{1}_Q ; \lambda,\nu\in(0,1), \lambda\neq\nu\}$.
\end{ex}

In the above notation, $I_Q$ is the identity matrix of size $Q$ and $
\mathbf{1}_Q$ is the size-$Q$ vector filled with $1$s.
Affiliation SBM is a simple two-parameters submodel of SBM commonly
used to detect communities with higher intra- than
inter-groups connectivities. It imposes as much symmetry on
elements of $\Pi_{\Q\calL}$ as allowed by
Assumption~\ref{hyp:identifiability} and constitutes the only model
where configuration equivalence (defined below) is confounded with
label-switching.

In less constrained models and as soon as $\Sig$ is not
reduced to identity, each permutation in $\Sig$
induces many different \emph{equivalent}
configurations. More precisely, for any $(s,t) \in\Sig$
and any $\bpi\in\Pi_{\Q\calL}$, we have
\begin{eqnarray*}
\mathbf{X}_{n,m}| \{\mathbf{Z}_n,\mathbf{W}_{m}
\} \displaystyle \mathop{=}^{d} \mathbf{X}_{n,m}| \bigl\{s(
\mathbf{Z}_n),t(\mathbf{W}_{m})\bigr\} ,\quad  \mbox{under
parameter value } \bpi,
\end{eqnarray*}
where $ \mathop{=}^{d} $ means equality in distribution.

\begin{remark}
In SBM with affiliation structure (see
Example~\ref{ex:affiliation}), the whole group of
permutations $\{(s,s) ; s \in\Sig_Q\}$ leaves the parameter set
$\Pi_{\Q\calL}$ invariant. For more general models, let us denote by
$[q,q']$ the transposition of $q$ and $q'$ in some set $\Q$.
We consider $(s,t)
=([q,q'],[l,l']) \in\mathfrak{S}_{\Q\calL}$.
Then any $\bpi\in
\Pi_{\Q\calL}$ satisfies
\begin{eqnarray*}
&& \forall i \in\Q\setminus\bigl\{q,q'\bigr\},\qquad  \pi_{il}=
\pi_{il'} ,
\\
&& \forall j \in\calL\setminus\bigl\{l,l'\bigr\},\qquad
\pi_{qj}=\pi_{q'j} ,
\\
&&\pi_{ql}=\pi_{q'l'}\quad  \mbox{and} \quad \pi_{q'l}=
\pi_{ql'}.
\end{eqnarray*}
In particular, for Assumption~\ref{hyp:identifiability} to be
satisfied while $([q,q'],[l,l'])$ belongs to $\Sig$
that leaves
$\Pi_{\Q\calL}$ invariant, it is necessary that either $\pi
_{ql}\neq
\pi_{ql'}$ or $\pi_{q'l}\neq\pi_{q'l'}$ (and then both
inequalities are satisfied).
\end{remark}

Note that the parameter sets $\Pi_{\Q\calL}$ that we consider are then
in a one-to-one correspondence with the subgroups $\Sig$. Note also
that we have $|\Sig|\le Q! L!$ (LBM) or $|\Sig|\le Q!$ (SBM).

We now define equivalent configurations in $\U$.

\begin{defi}\label{defi:equivalent_config}
Consider a parameter set $\Pi_{\Q\calL}$ invariant under the action
of some subgroup $\Sig$ of $\mathfrak{S}_{\Q\calL}$
and fix a parameter
value $\bpi\in\Pi_{\Q\calL}$. Any two groups configurations
$(\mathbf{z}_n,\mathbf{w}_m):=(z_{1},\ldots,z_n,w_{1},\ldots,w_m)$ and
$(\mathbf{z}^{\prime}_n,\mathbf{w}^{\prime}_m):=(z'_{1},\ldots
,z'_n,w'_{1},\ldots,w'_m)$ in $\U$
are called \emph{equivalent} (a relation denoted by $(\mathbf
{z}_n,\mathbf{w}_m) \sim
(\mathbf{z}^{\prime}_n,\mathbf{w}^{\prime}_m)$) if and only if
there exists $(s,t) \in\Sig$
such that
\begin{eqnarray*}
\bigl(s\bigl(\mathbf{z}^{\prime}_n\bigr),t\bigl(
\mathbf{w}^{\prime
}_m\bigr)\bigr):=\bigl(s\bigl(z_1'
\bigr),\ldots,s\bigl(z_n'\bigr),t\bigl(w_1'
\bigr),\ldots,t\bigl(w_m'\bigr)\bigr) =(
\mathbf{z}_n,\mathbf{w}_m).
\end{eqnarray*}
We let $\tU$ denote the quotient of $\U$ by this equivalence
relation. Note in particular that if $(\mathbf{z}_n,\mathbf{w}_m)
\sim(\mathbf{z}^{\prime}_n,\mathbf{w}^{\prime}_m)$
then for any $\bpi\in\Pi_{\Q\calL}$, we have
$(\pi_{z_iw_j})_{(i,j)\in\I} = (\pi_{z_i' w_j'})_{(i,j)\in
\I}$.
\end{defi}

For any vector $u=(u_1,\ldots,u_p) \in\mathbb{R}^p$, we let
$\|u\|_0:=\sum_{i=1}^p 1\{u_i\neq0\}$. The distance between two
different configurations $(\mathbf{z}_n,\mathbf{w}_m)\in\tU$ and
$(\mathbf{z}^{\prime}_n,\mathbf{w}^{\prime}_m)\in
\tU$ is
measured via the minimum $\|\cdot\|_0$ distance between any two
representatives of these classes. We thus let
%
\begin{equation}\label{eq:distance}
d\bigl((\mathbf{z}_n,\mathbf{w}_m),\bigl(
\mathbf{z}^{\prime}_n,\mathbf {w}^{\prime}_m
\bigr)\bigr) := \min\bigl\{ \bigl\|\mathbf{z}_n-s\bigl(
\mathbf{z}^{\prime}_n\bigr)\bigr\|_0 +\bigl\|\mathbf
{w}_m-t\bigl(\mathbf{w}^{\prime}_m\bigr)
\bigr\|_0 ; (s,t)\in\Sig\bigr\} .
\end{equation}
Note that this distance is well-defined on the space $\tU$. Note also
that when $\Sig$ is reduced to identity, the distance $d(\cdot,\cdot)$ is an
ordinary $\ell_0$ distance (up to a scale factor $2$ in SBM).

\subsection{Most likely configurations}

Among the set of all (up to equivalence) configurations $\tU$, we
shall distinguish some which are well-behaved in the following
sense. For any groups $q\in\Q$ and $l \in\calL$, consider the events
\begin{eqnarray*}
A_{q} = \Biggl\{\omega\in\Omega; N_{q}\bigl(
\mathbf{Z}_n(\omega)\bigr) :=\sum_{i=1}^n
1\bigl\{Z_i(\omega)=q\bigr\}< n \mu_{\min} / 2 \Biggr\} ,
\end{eqnarray*}
 and
 \begin{eqnarray*} B_{l} = \Biggl\{\omega\in\Omega; N_{l}
\bigl(\mathbf{W}_{m}(\omega)\bigr) :=\sum_{j=1}^m
1\bigl\{W_j(\omega)=l\bigr\}< m \mu _{\min} /2 \Biggr\} .
\end{eqnarray*}
Since $N_{q}(\mathbf{Z}_n) $ and $ N_{l}(\mathbf{W}_{m}) $ are sums
of i.i.d. Bernoulli
random variables with respective parameters $\alpha_{q}^\star$ and
$\beta_l^\star$, satisfying $\alpha_{q}^\star\wedge\beta_l^\star
\ge
\mu_{\min}$, a standard Hoeffding's Inequality gives
\[
\mathbb{P}_{\star}(A_{q}\cup B_l) \le\exp
\bigl[-n\bigl(\alpha_{q}^\star\bigr)^2/2\bigr]+
\exp\bigl[-m \bigl(\beta_{l}^\star\bigr)^2/2
\bigr]\le2 \exp\bigl[-(n \wedge m) \mu_{\min
}^2/2\bigr] .
\]
Taking an union bound, we obtain
\[
\mathbb{P}_{\star}\biggl(\displaystyle \bigcup_{(q,l) \in
\Q\times\calL} (A_{q}\cup
B_l)\biggr) \le \lleft\{ %
\begin{array} {l@{\qquad}l} 2QL \exp
\bigl[-(n \wedge m)\mu_{\min}^2/2\bigr] & \mbox{(LBM)},
\\\noalign{\vspace*{2pt}}
2Q \exp\bigl[-n \alpha_{\min}^2/2\bigr] & \mbox{(SBM)}.
\end{array} %
\rright.
\]
Now, consider the event $\Omega_0$ defined by
%
\begin{eqnarray}
\label{eq:Omega0} \Omega_0 &:=&\bigl\{\omega\in\Omega; \forall(q,l)\in\Q
\times \calL, N_{q}\bigl(\mathbf{Z}_n(\omega)\bigr) \ge n
\mu_{\min} / 2 \mbox{ and } N_{l}\bigl(\mathbf{W}_{m}(
\omega)\bigr) \ge m \mu_{\min} / 2\bigr\}
\nonumber
\\[-8pt]\\[-8pt]
&= &\bigcap_{(q,l) \in\Q\times\calL} (\bar{A}_{q}\cap\bar{B}_l),\nonumber
\end{eqnarray}
which has $\mathbb{P}_{\star}$-probability larger than $ 1 - 2QL \exp
[-(n \wedge m)
\mu_{\min}^2/2]$ (LBM) or larger than $ 1 - 2Q \exp[-n
\alpha_{\min}^2/2]$ (SBM) and its counterpart $\U^0$ defined by
%
\begin{equation}
\label{eq:good_set} \U^0 = \bigl\{ (\mathbf{z}_n,
\mathbf{w}_m) \in\U; \forall(q,l)\in\Q \times\calL, N_{q}(
\mathbf{z}_n) \ge n \mu_{\min} / 2 \mbox{ and }
N_{l}(\mathbf{w}_m) \ge m \mu_{\min} / 2 \bigr
\},
\end{equation}
where $N_{q}(\mathbf{z}_n) :=\sum_{i=1}^n 1\{z_i=q\}$ and $
N_{l}(\mathbf{w}_m)
$ is defined similarly. We extend this notation up to equivalent
configurations, by letting $\tU^0$ be the set of configurations
$(\mathbf{z}_n
,\mathbf{w}_m) \in\tU$ such that at least one (and then in fact all)
representative in the class belongs to $\U^0$.
Note that neither $N_q(\mathbf{z}_n)$ nor $N_l(\mathbf{w}_m)$ are
properly defined on
$\tU$, as these quantities may take different values for equivalent
configurations. However, as soon as one representative $(\mathbf
{z}_n,\mathbf{w}_m)$
belongs to $\U^0$, we both get $N_q(\mathbf{z}_n') \ge n \mu_{\min}
/ 2$ and
$N_l(\mathbf{w}_m') \ge m \mu_{\min} / 2$ for any $(\mathbf
{z}_n',\mathbf{w}_m')\sim(\mathbf{z}_n,
\mathbf{w}_m)$.
In the following, some properties will only be valid on the set of
configurations $\tU^0$.

\section{Groups posterior distribution}
\label{sec:post-distr}
\subsection{The groups posterior distribution}
We provide a preliminary lemma on the expression of the groups
posterior distribution.

\begin{lemma}\label{lem:posterior}
For any $ n,m\ge1$ and any $\theta\in\Theta$, the groups posterior
distribution writes for any $
(\mathbf{z}_n,\mathbf{w}_m)\in\U$,
%
\begin{eqnarray}
\label{eq:posterior} p_{n,m}^{\theta}(\mathbf{z}_n,
\mathbf{w}_m) &:=& \pr_\theta\bigl((\mathbf{Z}_n,
\mathbf{W}_{m})=(\mathbf{z}_n,\mathbf {w}_m)|
\mathbf{X}_{n,m}\bigr)
\nonumber
\\
&\propto&\lleft\{ %
\begin{array} {l@{\qquad}l}  \biggl(\displaystyle \prod
_{(i,j)\in\I} f(X_{ij};\pi_{z_i
w_j}) \biggr)
\Biggl( \displaystyle \prod_{i = 1}^n
\alpha_{z_i} \Biggr) \Biggl( \displaystyle \prod_{j=1}^m
\beta_{w_j} \Biggr) & \mbox{(LBM)},
\\\noalign{\vspace*{4pt}}
\biggl(\displaystyle \prod_{(i,j)\in\I} f(X_{ij};
\pi_{z_i
z_j}) \biggr) \Biggl( \displaystyle \prod_{i = 1}^n
\alpha_{z_i} \Biggr) & \mbox{(SBM)},
\end{array} %
\rright.
\end{eqnarray}
where $\propto$ means equality up to a normalizing constant.
\end{lemma}

The proof of this lemma is straightforward and therefore omitted.

In the following, we will consider the main term in the log ratio
$\log p_{n,m}^{\theta}(\mathbf{z}_n^{\star},\mathbf{w}_m^{\star})
-\log p_{n,m}^{\theta}(\mathbf{z}_n,\mathbf{w}_m) $ for two different
configurations $(\mathbf{z}_n^{\star},\mathbf{w}_m^{\star}),
(\mathbf{z}_n,\mathbf{w}_m) \in\U$. More precisely, we
introduce
%
\begin{equation}
\label{eq:delta} \forall\bigl(\mathbf{z}_n^{\star},
\mathbf{w}_m^{\star}\bigr), (\mathbf {z}_n,
\mathbf{w}_m) \in\tU, \qquad \delta^{\bpi}\bigl(\mathbf{z}_n^{\star},
\mathbf{w}_m^{\star},\mathbf {z}_n,
\mathbf{w}_m\bigr)= \sum_{(i,j)\in\I} \log
\biggl(\frac{f(X_{ij};\pi_{z^\star_iw^\star_j})}{ f(X_{ij};\pi
_{z_iw_j})} \biggr).
\end{equation}
Note that this quantity is well-defined on $\tU\times\tU$. We also consider
its expectation, under true parameter value $\theta^{\star}$ and
conditional on
the event $(\mathbf{Z}_n,\mathbf{W}_{m})=(\mathbf{z}_n^{\star
},\mathbf{w}_m^{\star})$; namely for any $ (\mathbf{z}_n^{\star
},\mathbf{w}_m^{\star})$ and
$(\mathbf{z}_n,\mathbf{w}_m) \in\tU$, we let
%
\begin{equation}
\label{eq:esp_delta} \Delta^{\bpi}\bigl(\mathbf{z}_n^{\star},
\mathbf{w}_m^{\star},\mathbf {z}_n,
\mathbf{w}_m\bigr)= \sum_{(i,j)\in\I}
\mathbb{E}_{\star} \biggl( \log \biggl(\frac{f(X_{ij};\pi_{z^\star_iw^\star_j})}{ f(X_{ij};\pi
_{z_iw_j})} \biggr) \bigl|(
\mathbf{Z}_n,\mathbf{W}_{m}) =\bigl(\mathbf{z}_n^{\star},
\mathbf {w}_m^{\star}\bigr) \biggr).
\end{equation}
Probabilities and expectations conditional on $(\mathbf{Z}_n,\mathbf
{W}_{m})=(\mathbf{z}_n^{\star},\mathbf{w}_m^{\star})$
and under parameter value $\theta^{\star}$ will be denoted by
$\mathbb{P}_{\star}^{\mathbf{z}_n^{\star}\mathbf{w}_m^{\star}}
$ and $\mathbb{E}_{\star}^{\mathbf{z}_n^{\star}\mathbf{w}_m^{\star
}}$, respectively.

\subsection{Assumptions on the model}
The results of this section are valid as long as the family of
distributions $\{f(\cdot;\pi); \pi\in\Pi\}$ satisfies some
properties. We thus formulate these as assumptions in this general
section, and establish later that these assumptions are satisfied in
each particular case to be considered.

The first of these assumptions is a (conditional on the configuration)
concentration inequality on the random variable
$\delta^{\bpi}(\mathbf{Z}_n,\mathbf{W}_{m},\mathbf{z}_n,\mathbf
{w}_m)$ around its conditional expectation.
We only require it to be valid for configurations
$(\mathbf{Z}_n,\mathbf{W}_{m})=(\mathbf{z}_n^{\star},\mathbf
{w}_m^{\star})\in\tU^0$. Note that under conditional probability
$\mathbb{P}_{\star}^{\mathbf{z}_n^{\star}\mathbf{w}_m^{\star}}$,
the random variables $\{X_{ij};
(i,j)\in\I\}$ are independent.

\begin{assumption}[(Concentration inequality)] \label{hyp:concentration}
Fix $(\mathbf{z}_n^{\star},\mathbf{w}_m^{\star})\in\tU^0 $
and $(\mathbf{z}_n,\mathbf{w}_m) \in\tU$ such that $(\mathbf
{z}_n,\mathbf{w}_m)\nsim(\mathbf{z}_n^{\star},\mathbf
{w}_m^{\star})$.
There exists some positive function $\psi^\star\dvtx  (0,+\infty)\to
(0,+\infty]$ such that for any $\bpi\in\Pi_{\Q\calL}$ and any
$\eps>0$, we have
%
\begin{eqnarray}
\label{eq:concentration} &&\mathbb{P}_{\star}^{\mathbf{z}_n^{\star}\mathbf{w}_m^{\star}} \bigl( \bigl|
\delta^{\bpi}\bigl(\mathbf{z}_n^{\star},\mathbf
{w}_m^{\star},\mathbf{z}_n,\mathbf{w}_m
\bigr) - \mathbb{E}_{\star}^{\mathbf{z}_n^{\star}\mathbf{w}_m^{\star}} \bigl(\delta^{\bpi}
\bigl(\mathbf{z}_n^{\star},\mathbf{w}_m^{\star
},
\mathbf{z}_n,\mathbf{w}_m\bigr) \bigr) \bigr| \ge
\eps(mr_1+ nr_2) \bigr)
\nonumber\\[-8pt]\\[-8pt]
&&\quad \le2\exp\bigl[- \psi^\star(\eps) (mr_1 +nr_2 )
\bigr] ,\nonumber
\end{eqnarray}
where the distance $d((\mathbf{z}_n^{\star},\mathbf{w}_m^{\star}) ,
(\mathbf{z}_n,\mathbf{w}_m)) $ defined
by \eqref{eq:distance} is attained for some permutations
$(s,t)\in\Sig$ and we set $r_1:=\|\mathbf{z}_n^{\star}-s(\mathbf
{z}_n)\|_0$ and
$r_2:=\|\mathbf{w}_m^{\star}-t(\mathbf{w}_m)\|_0$.
\end{assumption}

\begin{remark} \label{rem:concentration}
Assumption~\ref{hyp:concentration} is reasonable and is often obtained
by an exponential control of the centered random variable
\begin{eqnarray*}
Y_{\pi, \pi'} = \log \biggl( \frac{f(X; \pi)}{f(X; \pi')} \biggr)- \esp_{\pi}
\biggl[ \log \biggl( \frac{f(X; \pi)}{f(X; \pi')} \biggr) \biggr],
\end{eqnarray*}
uniformly in $\pi, \pi' \in\Pi$, where $\esp_{\pi}$ is the
expectation under $f( \cdot, \pi)$. As shown in Section~\ref{sec:scheme},
as soon as
\[
\psi_{\max}(\lambda) :=\sup_{\pi, \pi' \in\Pi}
\esp_{\pi} \bigl(\exp(\lambda Y_{\pi, \pi'})\bigr)
\]
is finite for $\lambda$
in a small open interval $I \subset\mathbb{R}$ around $0$, a
Cramer--Chernoff bound shows that Inequality (\ref{eq:concentration}) is
satisfied with
\begin{eqnarray*}
\psi^\star(\eps) :=\frac{\mu^2_{\min}}{8} \sup_{\lambda\in
I}
\bigl(\lambda\eps- \psi_{\max}(\lambda)\bigr).
\end{eqnarray*}
\end{remark}

The second assumption needed is a bound on the Kullback--Leibler
divergences for elements of the family $\{f(\cdot;\pi) ; \pi\in
\Pi\}$. We let
%
\begin{equation}
\label{eq:Kullback} D\bigl(\pi\parallel \pi'\bigr) :=\int_{\X}
\log \biggl(\frac
{f(x;\pi)}{f(x;\pi')} \biggr) f(x;\pi)\,\mathrm{d}x .
\end{equation}

\begin{assumption}[(Bounds on Kullbak--Leibler
divergences)] \label{hyp:Bound_KL}
We assume that
\begin{eqnarray*}
\kappa_{\max} :=\max\bigl\{D\bigl(\pi \parallel \pi'\bigr) ; \pi,
\pi' \in\Pi\bigr\} <+ \infty.
\end{eqnarray*}
\end{assumption}

Note that $\kappa_{\max} <+\infty$ is automatically satisfied when the
distributions in the family $\{f(\cdot;\pi) $; $ \pi\in\Pi\}$ form
an exponential family with natural parameter $\pi$ belonging to a
compact set $\Pi$.
In particular, this is not the case for Bernoulli distributions when
we authorize some probabilities $\pi$ to be $0$ or $1$, as the
corresponding natural parameter then takes the values $-\infty$ and
$+\infty$. In the
following, we thus exclude for the binary case the possibility that
classes may be almost
never or almost surely connected. We also introduce
%
\begin{equation}
\label{eq:kappa_min} \kappa_{\min}= \kappa_{\min}\bigl(
\bpi^\star\bigr) :=\min \bigl\{D\bigl(\pi^\star_{ql}
\parallel \pi^\star_{q'l'}\bigr) ; (q,l), \bigl(q',l'
\bigr)\in\Q \times \calL,\pi^\star_{ql}\neq
\pi^\star_{q'l'}\bigr\} >0,
\end{equation}
where positivity is a consequence of
Assumption~\ref{hyp:identifiability}. The parameter $\kappa_{\min}$
measures how far apart the non-identical entries of $\boldsymbol{\pi
}^{\star}$ are and is
the main driver of the convergence rate of the posterior
distribution. Note that the Kullback--Leibler
divergence captures the differences between the distributions and
not only their mean values. As we
already mentioned in the \hyperref[intro]{Introduction}, this is in contrast to results
as in \cite{Flynn_Perry} and we may for instance recover groups that
differ only
in their variance.

The last assumption needed is a Lipschitz condition on an integrated
version of the function $\pi\mapsto\log f(x;\pi)$.

\begin{assumption}
\label{hyp:Lipschitz}
There exists some positive constant $L_0$ such that for any
$\bpi,\bpi'\in\Pi_{\Q\calL}$ and any $(q,l),(q',l')\in\Q\times
\calL$, we have
\begin{eqnarray*}
\biggl\llvert \int_{\X} \log\frac{f(x;\pi_{ql})}{f(x;\pi'_{ql})} f(x;
\pi_{q'
l'})\, \mathrm{d}x \biggr\rrvert \le L_0\bigl\|\bpi-\bpi'
\bigr\|_{\infty} .
\end{eqnarray*}
\end{assumption}

\begin{remark} \label{rem:exponential_families} As illustrated in
Section~\ref{sec:comm-expon-famil}, many exponential
families satisfy Assumptions \ref{hyp:concentration} to
\ref{hyp:Lipschitz} as long as the \emph{natural parameter} of that
family (e.g., $\log(p)$ for Poisson distribution or
$\log(p/(1-p))$ for the binomial) is restricted to a compact
set. This includes but is not limited to Gaussian (location or
scale model), Poisson, binary, binomial and multinomial
distributions.
\end{remark}

\subsection{Convergence of the posterior distribution}

We now establish some preliminary results. The first one gives the
behavior of the conditional expectation $\Delta^{\bpi}$ defined
by \eqref{eq:esp_delta} with respect to the distance between the two
configurations $(\mathbf{Z}_n,\mathbf{W}_{m})$ and $(\mathbf
{z}_n,\mathbf{w}_m)$.

\begin{prop}[(Behavior of conditional expectation)]
\label{prop:CondExp_Order}
Under Assumptions \ref{hyp:identifiability}, \ref{hyp:symetrie}
and \ref{hyp:Bound_KL}, the constant $C=2\kappa_{\max} >0$ is such
that for any parameter value $\bpi\in\Pi_{\Q\calL}$ and any
configuration $(\mathbf{z}_n,\mathbf{w}_m) \in\tU$,
we have $\mathbb{P}_{\star}$-almost surely
%
\begin{equation}
\label{eq:CondExp_Order_Sup} \mathbb{E}_{\star}^{\mathbf{Z}_n\mathbf{W}_{m}} \bigl(
\delta^{\bpi
}(\mathbf{Z}_n,\mathbf{W}_{m},
\mathbf{z}_n,\mathbf{w}_m) \bigr) \le\frac{C} 2(mr_1+
nr_2) ,
\end{equation}
where the distance $d((\mathbf{Z}_n,\mathbf{W}_{m}),(\mathbf
{z}_n,\mathbf{w}_m)) $ is attained for some
$(s,t)\in\Sig$ and we set $r_1:=\|\mathbf{Z}_n-s(\mathbf{z}_n)\|_0$ and
$r_2:=\|\mathbf{W}_{m}-t(\mathbf{w}_m)\|_0$.

Furthermore, under additional Assumption~\ref{hyp:Lipschitz}, the
constant $c=\mu_{\min}^2\kappa_{\min}/16$ is such that
on the set $\Omega_0$ defined by \eqref{eq:Omega0} whose
$\mathbb{P}_{\star}$-probability satisfies
\[
\lleft\{ %
\begin{array} {l@{\qquad}l} \mathbb{P}_{\star}(
\Omega_0) \ge1 - 2QL \times \exp\bigl[-(n \wedge m)
\mu_{\min}^2/2\bigr] & \mbox{(LBM)},
\\\noalign{\vspace*{2pt}}
\mathbb{P}_{\star}(\Omega_0) \ge1 - 2Q \times \exp\bigl[-n
\alpha_{\min}^2/2\bigr]  & \mbox{(SBM)}, \end{array} %
\rright.
\]
for any parameter value $\bpi\in\Pi_{\Q\calL}$ and any sequence
$(\mathbf{z}_n,\mathbf{w}_m) \in\tU$, we have
%
\begin{equation}
\label{eq:CondExp_Order_Inf} \mathbb{E}_{\star}^{\mathbf{Z}_n\mathbf{W}_{m}} \bigl(
\delta^{\bpi
}(\mathbf{Z}_n,\mathbf{W}_{m},
\mathbf{z}_n,\mathbf{w}_m) \bigr) \ge2\bigl(c -
L_0\bigl\|\bpi-\boldsymbol{\pi}^{\star}\bigr\|_{\infty}\bigr)
(mr_1 + nr_2) . 
\end{equation}
\end{prop}

\begin{pf}
Note that
\begin{eqnarray*}
\mathbb{E}_{\star}^{\mathbf{Z}_n\mathbf{W}_{m}} \bigl(\delta^{\bpi
}(
\mathbf{Z}_n,\mathbf{W}_{m},\mathbf{z}_n,
\mathbf{w}_m) \bigr) = \sum_{(\mathbf{z}_n^{\star},\mathbf{w}_m^{\star}) \in\tU
}
\mathbb{E}_{\star}^{\mathbf{z}_n^{\star}\mathbf{w}_m^{\star}} \bigl(\delta ^{\bpi}\bigl(
\mathbf{z}_n^{\star},\mathbf{w}_m^{\star},
\mathbf {z}_n,\mathbf{w}_m\bigr) \bigr)\times
1_{ (\mathbf{Z}_n,\mathbf{W}_{m})= (\mathbf{z}_n^{\star},\mathbf
{w}_m^{\star})} ,
\end{eqnarray*}
so that we can work on the set $\{ (\mathbf{Z}_n,\mathbf{W}_{m})=
(\mathbf{z}_n^{\star},\mathbf{w}_m^{\star})\} $ for a
fixed configuration $(\mathbf{z}_n^{\star},\mathbf{w}_m^{\star})\in
\tU$. Moreover, we can choose
$(\mathbf{z}_n,\mathbf{w}_m)\in\tU$ that realizes the distance
$d((\mathbf{z}_n^{\star}, \mathbf{w}_m^{\star}),(\mathbf{z}_n
, \mathbf{w}_m))$, namely such that $d((\mathbf{z}_n^{\star},
\mathbf{w}_m^{\star}),(\mathbf{z}_n, \mathbf{w}_m)) = \|\mathbf
{z}_n^{\star}
-\mathbf{z}_n\|_0+ \|\mathbf{w}_m^{\star}-\mathbf{w}_m\|_0=r_1+r_2$.

If $(\mathbf{z}_n,\mathbf{w}_m) = (\mathbf{z}_n^{\star},\mathbf
{w}_m^{\star})$, namely $r_1=r_2=0$, then we have
$\delta^{\bpi}(\mathbf{z}_n^{\star},\mathbf{w}_m^{\star},\mathbf
{z}_n,\mathbf{w}_m)=0$
and the lemma is proved. Otherwise, we may have
$r_1$ or $r_2$ equal to zero but $r_1+r_2\ge1$. Without loss of
generality, we can assume that $\mathbf{z}_n^{\star}$ and $\mathbf
{z}_n$ (respectively $\mathbf{w}_m^{\star}$
and $\mathbf{w}_m$) differ at the first $r_1$ (resp. $r_2$)
indexes.

First, let us note that
%
\begin{equation}
\label{eq:esp_cond} \mathbb{E}_{\star}^{\mathbf{z}_n^{\star}\mathbf{w}_m^{\star}} \bigl(
\delta^{\bpi}\bigl(\mathbf{z}_n^{\star},
\mathbf{w}_m^{\star
},\mathbf{z}_n,
\mathbf{w}_m\bigr) \bigr) = \sum_{(i,j)\in\tilde\I}
\int_{\X} \log \biggl(\frac{f(x;\pi
_{z_i^\star
w_j^\star} )}{f(x;\pi_{z_i w_j} )} \biggr) f\bigl(x;
\pi^\star_{z_i^\star
w_j^\star}\bigr) \,\mathrm{d}x ,
\end{equation}
where $\tilde\I= \I\smallsetminus\{(i,j) ; i>r_1 \mbox{ and } j
>r_2\}$. This leads to
\begin{eqnarray*}
\mathbb{E}_{\star}^{\mathbf{z}_n^{\star}\mathbf{w}_m^{\star}} \bigl(\delta^{\bpi}\bigl(
\mathbf{z}_n^{\star},\mathbf{w}_m^{\star
},
\mathbf{z}_n,\mathbf{w}_m\bigr) \bigr) \le(m
r_1+nr_2-r_1r_2)
\kappa_{\max} \le\frac{C} 2 (mr_1 + nr_2)
,
\end{eqnarray*}
with $C = 2\kappa_{\max}$,
which establishes Inequality \eqref{eq:CondExp_Order_Sup}.

To prove Inequality \eqref{eq:CondExp_Order_Inf}, we write the
decomposition
%
\begin{eqnarray}
\label{eq:decomp_exKullback} &&\sum_{(i,j)\in\tilde\I} \int_{\X}
\log \biggl(\frac{f(x;\pi
_{z_i^\star
w_j^\star} )}{f(x;\pi_{z_i w_j} )} \biggr) f\bigl(x;\pi^\star_{z_i^\star
w_j^\star}
\bigr) \,\mathrm{d}x\nonumber \\
&&\quad = \sum_{(i,j)\in\tilde\I} \biggl\{ - D\bigl(
\pi^\star_{z_i^\star w_j^\star} \parallel \pi_{z_i^\star w_j^\star}\bigr)+ D\bigl(\pi^\star_{z_i^\star w_j^\star} \parallel\pi^\star_{z_iw_j}
\bigr)
\\
&&\hphantom{q=\sum_{(i,j)\in\tilde\I} \biggl\{}{} + \int_{\X} \log\frac{f(x;\pi^\star_{z_iw_j})}{f(x;\pi
_{z_iw_j})} f\bigl(x;
\pi^\star_{z_i^\star
w_j^\star}\bigr) \,\mathrm{d}x \biggr\}.\nonumber
\end{eqnarray}
According to Assumption~\ref{hyp:Lipschitz}, the third term in the
right-hand side of the above equation is lower-bounded by
$- L_0\|\bpi-\boldsymbol{\pi}^{\star}\|_{\infty
}(mr_1+nr_2-r_1r_2)$. The first term in
this right-hand side
is handled similarly as we have
\begin{eqnarray*}
0&<& \sum_{(i,j)\in\tilde\I} D\bigl(\pi^\star_{z_i^\star w_j^\star}
\parallel \pi_{z_i^\star w_j^\star}\bigr) = \sum_{(i,j)\in\tilde\I} \int
_{\X
} \log \frac{f(x;\pi^\star_{z_i^\star w_j^\star})}{f(x;\pi_{z_i^\star
w_j^\star})} f\bigl(x; \pi^\star_{z_i^\star w_j^\star}
\bigr) \,\mathrm{d}x
\\
&\le& L_0\bigl\|\bpi-\boldsymbol{\pi}^{\star}\bigr\|_{\infty}(mr_1+nr_2-r_1r_2),
\end{eqnarray*}
where the second inequality is another application of
Assumption~\ref{hyp:Lipschitz}.

The central term appearing in the right-hand side of
decomposition \eqref{eq:decomp_exKullback} is handled relying on the
next lemma, whose proof is postponed to the \hyperref[sec:appendix]{Appendix}.
It is a generalization to LBM of Proposition B.5 in
\cite{Alain_JJ_LP} that considers SBM only. This lemma bounds from
below the number of pairs $(i,j)$ such that
\[
\pi^\star_{z_i^\star w_j^\star} \neq\pi^\star_{z_i w_j}
\]
and establishes that it is of order $m r_1+ nr_2$. This is
possible only for the configurations $(\mathbf{z}_n^{\star},\mathbf
{w}_m^{\star})\in\tU^0$ defined
by \eqref{eq:good_set}. For the rest of the proof, we work on the set
$\Omega_0$, meaning that
we assume $\{(\mathbf{Z}_n,\mathbf{W}_{m})=(\mathbf{z}_n^{\star
},\mathbf{w}_m^{\star}) \in\tU^0\}$.

\begin{lemma}[(Bound on the number of differences)]
\label{lem:Bound_Number}
Under Assumptions \ref{hyp:identifiability} and \ref{hyp:symetrie},
for any
configurations $(\mathbf{z}_n,\mathbf{w}_m) \in\tU$ and $(\mathbf
{z}_n^{\star},\mathbf{w}_m^{\star}) \in\tU^0$,
we have
%
\begin{equation}
\label{eq:Bound_Number} \diff \bigl(\mathbf{z}_n,\mathbf{w}_m,
\mathbf{z}_n^{\star
},\mathbf{w}_m^{\star}
\bigr) := \bigl| \bigl\{ (i,j) \in\I; \pi^\star_{z_i w_j} \neq
\pi^\star_{z_i^\star w_j^\star} \bigr\}\bigr |\ge\frac{\mu_{\min
}^2}{8}
(mr_1 + nr_2),
\end{equation}
where the distance $d(
(\mathbf{z}_n,\mathbf{w}_m),(\mathbf{z}_n^{\star},\mathbf
{w}_m^{\star})) $ is attained for some permutations
$(s,t)\in\Sig$ and we set $r_1:=\|\mathbf{z}_n-s(\mathbf
{z}_n^{\star})\|_0$ and $r_2:=\|\mathbf{w}_m
-t(\mathbf{w}_m^{\star})\|_0$.
\end{lemma}

According to Assumption~\ref{hyp:Bound_KL}, if $\pi^\star_{z_i w_j}
\neq
\pi^\star_{z^\star_i w^\star_j}$, the divergence $D(\pi^\star
_{z_i^\star w_j^\star} \parallel \pi^\star_{z_i w_j})$ is at least $\kappa
_{\min}$. We thus get
\begin{eqnarray*}
\sum_{(i,j)\in\tilde\I} D\bigl(\pi^\star_{z_i^\star w_j^\star}
\parallel\pi^\star_{z_iw_j}\bigr) \ge\frac{\mu_{\min}^2\kappa_{\min}}{8}
(mr_1 + nr_2).
\end{eqnarray*}
Coming back to \eqref{eq:decomp_exKullback} and \eqref{eq:esp_cond},
we obtain
\begin{eqnarray*}
&&\sum_{(i,j)\in\tilde\I} \int_{\X} \log
\biggl(\frac{f(x;\pi
_{z_i^\star
w_j^\star} )}{f(x;\pi_{z_i w_j} )} \biggr) f\bigl(x;\pi^\star_{z_i^\star
w_j^\star}\bigr)
\,\mathrm{d}x \\
&&\quad \ge \biggl(\frac{\mu_{\min}^2\kappa_{\min}}{8} - 2 L_0\bigl\|\bpi-\boldsymbol{
\pi}^{\star}\bigr\|_{\infty} \biggr) (mr_1 +
nr_2)
\end{eqnarray*}
and thus conclude
\begin{eqnarray*}
\mathbb{E}_{\star}^{\mathbf{z}_n^{\star}\mathbf{w}_m^{\star}} \bigl(\delta^{\bpi}\bigl(
\mathbf{z}_n^{\star},\mathbf{w}_m^{\star
},
\mathbf{z}_n,\mathbf{w}_m\bigr) \bigr) \ge \biggl(
\frac{\mu_{\min}^2\kappa_{\min}}{8} - 2 L_0\bigl\|\bpi-\boldsymbol{\pi}^{\star}
\bigr\|_{\infty} \biggr) (mr_1 + nr_2) .
\end{eqnarray*}
By letting $c$ = $\mu_{\min}^2\kappa_{\min}/16$, we obtain exactly
\eqref{eq:CondExp_Order_Inf}.
\end{pf}

In the following, we will consider asymptotic results where both $n$ and
$m$ increase to infinity. The next assumption settles the
relative rates of convergence of $n$ and $m$ in LBM. With no loss of
generality, we assume in the following that $n\ge m$, view $m=m_n$ as
a sequence depending on $n$ and state the convergence results with
respect to $n\to+\infty$. Note that the assumption is trivial for SBM.

\begin{assumption}[(Asymptotic setup)] \label{hyp:asymptotics}
The sequence $(m_n)_{n\ge1} $ converges to infinity under the
constraints $m_n\le n$ and $(\log n)/m_n \to0$.
\end{assumption}

We now state the main theorem.
%
\begin{thm} \label{thm:cv_posterior} Under
Assumptions \ref{hyp:identifiability} to \ref{hyp:asymptotics},
following the notation of Proposition~\ref{prop:CondExp_Order}, for
any $\eta\in(0,c/(2L_0))$, there exists a family $\{\eps_{n,m}
\}_{n,m}$ of positive real numbers with $\sum_{n} \eps_{n,m_n}
<+\infty$, such that on a set $\Omega_1$ whose $\mathbb{P}_{\star
}$-probability is
at least $1 - \eps_{n,m}$ and for any $\theta=(\bmu,\bpi)\in\Theta$
satisfying $\|\bpi-\bpi^\star\|_\infty\le\eta$, we have for any
$(\mathbf{z}_n,\mathbf{w}_m)\in\U$ and any $(s,t) \in\Sig$,
%
\begin{eqnarray}
\label{eq:cv_post_ratio_1} && \log\frac{p_{n,m}^{\theta}(s(\mathbf{Z}_n),t(\mathbf
{W}_{m}))}{p_{n,m}^{\theta}(\mathbf{z}_n,\mathbf{w}_m)}\nonumber
\\[-8pt]\\[-8pt]
&&\quad  \ge\lleft\{ %
\begin{array} {l} \bigl(c-2L_0\bigl\|\bpi-
\boldsymbol{\pi}^{\star}\bigr\|_\infty\bigr) (mr_1 +
nr_2) - K\bigl(\bigl\| s(\mathbf{Z}_n)-\mathbf{z}_n
\bigr\|_0+\bigl\|t(\mathbf{W}_{m})-\mathbf {w}_m
\bigr\|_0\bigr) \\
\qquad  \mbox{(LBM)},
\\\noalign{\vspace*{+4pt}}
\bigl(c-2L_0\bigl\|\bpi-\boldsymbol{\pi}^{\star}\bigr\|_\infty
\bigr) 2nr_1 - K\bigl\| s(\mathbf{Z}_n)-\mathbf{z}_n
\bigr\|_0 \\
\qquad  \mbox{(SBM)}, \end{array} %
\nonumber
\rright.
\end{eqnarray}
 and
\begin{eqnarray}\label{eq:cv_post_ratio_2}
&& \log\frac{p_{n,m}^{\theta}(s(\mathbf
{Z}_n),t(\mathbf{W}_{m}))}{p_{n,m}^{\theta}(\mathbf{z}_n,\mathbf{w}_m)}
\nonumber\\[-8pt]\\[-8pt]
&&\quad \le\lleft\{ %
\begin{array} {l@{\qquad}l} C (mr_1 +
nr_2) + K \bigl(\bigl\| s(\mathbf{Z}_n)-\mathbf{z}_n
\bigr\|_0+\bigl\|t(\mathbf {W}_{m})-\mathbf{w}_m
\bigr\|_0\bigr) & \mbox{(LBM)},
\\\noalign{\vspace*{2pt}}
C 2nr_1 + K \bigl\| s(\mathbf{Z}_n)-\mathbf{z}_n
\bigr\|_0 & \mbox{(SBM)}, \end{array} %
\nonumber
\rright.
\end{eqnarray}
where the distance $d((s(\mathbf{Z}_n),t(\mathbf{W}_{m})) , (\mathbf
{z}_n,\mathbf{w}_m))$, which does not
depend on $(s,t)$, is attained for some permutation $(\tilde{s},
\tilde{t})\in\Sig$ and we set $r_1:=\|\mathbf{Z}_n-\tilde
{s}(\mathbf{z}_n)\|_0$ and
$r_2:=\|\mathbf{W}_{m}-\tilde{t}(\mathbf{w}_m)\|_0$ and
$K=\log(\alpha_{\max}/\alpha_{\min}) \vee
\log(\beta_{\max}/\beta_{\min})$.
\end{thm}

Let us comment this result.
Inequalities \eqref{eq:cv_post_ratio_1} and \eqref{eq:cv_post_ratio_2}
provide a control of the concentration of the posterior distribution
on the actual (random) configuration $(\mathbf{Z}_n,\mathbf{W}_{m})$,
\emph{viewed as an
equivalence class in $\tU$}. The most important one
is \eqref{eq:cv_post_ratio_1} that provides a lower bound on the
posterior probability of any configuration equivalent to the actual
configuration $(\mathbf{Z}_n,\mathbf{W}_{m})$ compared to any other
configuration
$(\mathbf{z}_n,\mathbf{w}_m)$. In this inequality, two different
distances appear
between these configurations, namely the $\ell_0$ distance and the
distance $d(\cdot,\cdot)$ given by \eqref{eq:distance}, on the set of actual
configurations (so that $d(\cdot,\cdot)$ is linked with the parameter $\bpi$ and
its symmetries). When the subgroup $\Sig$ is reduced to identity (no
symmetries allowed in $\bpi$), these two distances coincide and the
statement substantially simplifies. Another case where it simplifies
is when $K=0$, corresponding to $\alpha_{\max}=\alpha_{\min}$ and
$\beta_{\max}=\beta_{\min}$ or equivalently to uniform group
proportions. These two particular cases are further expanded below in
the first two corollaries. In general, the two different distances
appear and play a different role in this inequality. In particular,
consider Inequality \eqref{eq:cv_post_ratio_1} with for instance
$s=\mathit{Id}=t$. It may be the case that a putative configuration
$(\mathbf{z}_n,\mathbf{w}_m)$ is equivalent to the actual random one
$(\mathbf{Z}_n,\mathbf{W}_{m})$ in the
sense of relation $\sim$, and thus their distance $d(\cdot,\cdot)$ is zero
($r_1=r_2=0$ above), but their $\ell_0$ distance is large. Then, the
posterior distribution $p_{n,m}^{\theta}$ will not concentrate on
$(\mathbf{z}_n,\mathbf{w}_m)$ due
to the existence of different group proportions $\bmu$ that help
distinguish between $(\mathbf{Z}_n,\mathbf{W}_{m})$ and this
equivalent configuration
$(\mathbf{z}_n,\mathbf{w}_m)$. The extent to which the group
proportions $\bmu$ are
different is measured by $K= \log(\alpha_{\max}/\alpha_{\min})
\vee
\log(\beta_{\max}/\beta_{\min})$. When this quantity is small
compared to the term $c-2L_0\eta$ (depending on $\bpi$, the
connectivity part of the parameter) appearing in
\eqref{eq:cv_post_ratio_1}, the term $K(\| \mathbf{Z}_n-\mathbf
{z}_n\|_0+\|\mathbf{W}_{m}-\mathbf{w}_m\|_0)
$ is negligible and the posterior distribution $p_{n,m}^{\theta}$ will not
distinguish between the actual configuration and any equivalent one.

Before giving the proof of the theorem, we provide some corollaries
that will help understand the importance of the previous result. The
first two corollaries deal with special setups and the third one is an
attempt to give a general understanding of the behaviour of the groups
posterior distribution. All these results state that, under some
appropriate assumptions, the posterior distribution $p_{n,m}^{\theta}$
concentrates on the actual random configuration $(\mathbf{Z}_n,\mathbf
{W}_{m})$, with
large probability. We stress the fact that the results are valid for
any parameter value $\theta$ (satisfying some additional assumption)
and not only the true one $\theta^\star$. More precisely, the results
are valid at any $\theta=(\bmu,\bpi)$ such that $\bpi$ is close enough
to the true value $\boldsymbol{\pi}^{\star}$.

\begin{cor}[(Case $\Sig=\{(\mathit{Id},\mathit{Id})\}$)] \label{cor:1} Under
Assumptions \ref{hyp:identifiability} to \ref{hyp:asymptotics} and
when $\Sig=\{(\mathit{Id},\mathit{Id})\}$, we obtain that on the set $\Omega_1$ whose
$\mathbb{P}_{\star}$-probability is at least $1 - \eps_{n,m}$, for
any parameter
$\theta=(\bmu,\bpi)\in\Theta$ satisfying
$\|\bpi-\bpi^\star\|_\infty\le\eta$ for small enough $\eta$, we
have
\begin{eqnarray*}
p_{n,m}^{\theta}(\mathbf{Z}_n,
\mathbf{W}_{m}) \ge1 - a_{n,m} \exp(a_{n,m}) \quad \mbox{and}\quad  p_{n,m}^{\theta}(\mathbf{Z}_n,
\mathbf{W}_{m}) \le\bigl(1 + b_{n,m}\mathrm{e}^{b_{n,m}}
\bigr)^{-1} ,
\end{eqnarray*}
where
%
\begin{equation}
\label{eq:anm} \lleft\{ %
\begin{array} {l@{\qquad}l} a_{n,m} =
\bigl(n\mathrm{e}^{-(c-2L_0\eta)m+K } + m\mathrm{e}^{-(c-2L_0\eta)n+K }\bigr) ;\\\noalign{\vspace*{2pt}}
 b_{n,m} = \bigl(n
\mathrm{e}^{-Cm-K} + m \mathrm{e}^{-Cn-K}\bigr) & \mbox{(LBM)},
\\\noalign{\vspace*{4pt}}
a_{n,n}= n\mathrm{e}^{-2n(c-2L_0\eta) +K} ; \qquad  b_{n,n} = n \mathrm{e}^{-2Cn-K}
& \mbox{(SBM)}, \end{array} %
\rright.
\end{equation}
all converge to $0$ as $n\to+\infty$. As a consequence, relying on
the maximum a posteriori (MAP) procedure, at a parameter value $\hat
{\theta}=(\hat{\bmu},\hat{\bpi})$ such that $\hat{\bpi}$ converges to the
true parameter value $\boldsymbol{\pi}^{\star}$, namely
\begin{eqnarray*}
(\widehat{\mathbf{Z}}_n,\widehat{\mathbf{W}}_{m}):=
\displaystyle \mathop{\argmax}_{(\mathbf
{z}_n,\mathbf{w}_m) \in\U} p_{n,m}^{\hat{\theta}} (
\mathbf{z}_n,\mathbf{w}_m) , \qquad \mbox{where } \hat{\theta}=(
\hat{\bmu} ,\hat{\bpi}) \mbox{ and } \hat{\bpi}\to\boldsymbol{\pi}^{\star}
\end{eqnarray*}
the number of misclassified rows and/or columns on the set
$\Omega_1$
\begin{eqnarray*}
\sum_{i=1}^n 1\{\hat Z_i
\neq Z_i\} + \sum_{j=1}^m 1\{
\hat W_j\neq W_j\} \qquad \mbox{(LBM)}\quad  \mbox{or}\quad  \sum
_{i=1}^n 1\{\hat Z_i\neq
Z_i\} \qquad \mbox{(SBM)},
\end{eqnarray*}
is exactly 0 for large enough $n$.
\end{cor}

\begin{cor}[(Case of uniform group proportions)] \label{cor:2} Under
Assumptions \ref{hyp:identifiability} to \ref{hyp:asymptotics} and
when $K=0$, we obtain that on the set $\Omega_1$, for any parameter
$\theta=(\bmu,\bpi)\in\Theta$ satisfying
$\|\bpi-\bpi^\star\|_\infty\le\eta$ for small enough $\eta$, we have
\begin{eqnarray*}
p_{n,m}^{\theta} \bigl( \bigl\{ (
\mathbf{z}_n, \mathbf{w}_m)\in\U; (\mathbf{z}_n,
\mathbf{w}_m) \sim(\mathbf{Z}_n,\mathbf{W}_{m})
\bigr\} \bigr) \ge1 - |\Sig| a_{n,m}\mathrm{e}^{a_{n,m}}
\end{eqnarray*}
 and
 \begin{eqnarray*} p_{n,m}^{\theta} \bigl( \bigl\{ (
\mathbf{z}_n, \mathbf{w}_m)\in\U; (\mathbf{z}_n,
\mathbf{w}_m) \sim(\mathbf{Z}_n,\mathbf{W}_{m})
\bigr\} \bigr) \le\bigl(1 + |\Sig|b_{n,m}\mathrm{e}^{b_{n,m}}
\bigr)^{-1},
\nonumber
\end{eqnarray*}
where $a_{n,m}$ and $b_{n,m}$ are defined through \eqref{eq:anm}
with $K=0$
and converge to $0$ as $n\to+\infty$.
Moreover,
\begin{eqnarray*}
p_{n,m}^{\theta} ( \mathbf{Z}_n,
\mathbf{W}_{m} ) = \frac{1}{|\Sig|} p_{n,m}^{\theta}
\bigl( \bigl\{ (\mathbf{z}_n, \mathbf {w}_m)\in\U; (
\mathbf{z}_n, \mathbf{w}_m) \sim (\mathbf{Z}_n,
\mathbf{W}_{m}) \bigr\} \bigr) .
\end{eqnarray*}
\end{cor}

\begin{cor}[(General case)] \label{cor:3} Under
Assumptions \ref{hyp:identifiability} to \ref{hyp:asymptotics}, we
obtain that on the set $\Omega_1$, for any parameter
$\theta=(\bmu,\bpi)\in\Theta$ satisfying
$\|\bpi-\bpi^\star\|_\infty\le\eta$ for small enough $\eta$, we
have
\begin{eqnarray*}
p_{n,m}^{\theta} \bigl( \bigl\{ (
\mathbf{z}_n, \mathbf{w}_m)\in\U; (\mathbf{z}_n,
\mathbf{w}_m) \sim(\mathbf{Z}_n,\mathbf{W}_{m})
\bigr\} \bigr) \ge 1 - |\Sig| a_{n,m}\mathrm{e}^{a_{n,m}}
\end{eqnarray*}
 and
 \begin{eqnarray*} p_{n,m}^{\theta} \bigl( \bigl\{ (
\mathbf{z}_n, \mathbf{w}_m)\in\U; (\mathbf{z}_n,
\mathbf{w}_m) \sim(\mathbf{Z}_n,\mathbf{W}_{m})
\bigr\} \bigr) \le \bigl(1 + |\Sig| b_{n,m}\mathrm{e}^{b_{n,m}}
\bigr)^{-1} ,
\nonumber
\end{eqnarray*}
where $a_{n,m}$ and $b_{n,m}$ are defined through \eqref{eq:anm}
and converge to $0$ as $n\to+\infty$.
\end{cor}

\begin{remark} Theorem~\ref{thm:cv_posterior} and
Corollaries \ref{cor:1} to \ref{cor:3} are expressed in full
generality but their results apply both to LBM and SBM with the
notation adopted in Section~\ref{sec:model}. In particular, the
expressions given in these statements simplify for SBM as $n = m,
\mathbf{Z}_n=\mathbf{W}_{m}, s=t$ and $r_1=r_2$.
\end{remark}

\begin{remark}
Note that the convergence of the posterior distribution (to the set
of configurations equivalent to the actual random one) happens at a
rate determined by the constant
\begin{eqnarray*}
c-2L_0\eta>0.
\end{eqnarray*}
Typically, the rate of this convergence is fast when $\bpi$ is not
too different from $\boldsymbol{\pi}^{\star}$ (namely $\|\bpi
-\boldsymbol{\pi}^{\star}\|_{\infty}$ and
thus $L_0 \eta$ small) while the connectivity parameters are
sufficiently distinct (namely $\kappa_{\min}$ and thus $c$ large).

When $\Sig=\{(\mathit{Id},\mathit{Id})\}$, the actual configuration has no other
equivalent one and the posterior distribution converges to it. When
$K=0$, group proportions are equal and do not discriminate between
equivalent configurations. Therefore, all equivalent configurations
(if any) are equally likely. In all other cases, the support of the
posterior distribution converges to the set of configurations
equivalent to the actual one, including the actual one. However, the
latter may not be the most likely among those. Provided $n$ and $m$
are large enough, the most likely configuration is the configuration
$(\mathbf{z}_n, \mathbf{w}_m)$ equivalent to $(\mathbf{Z}_n, \mathbf
{W}_{m})$ which maximizes the quantity
\begin{eqnarray*}
\sum_{i = 1}^n \log{\alpha_{z_i}}
+ \sum_{j = 1}^m \log {
\beta_{w_j}} = \sum_{q = 1}^Q
N_q(\mathbf{z}_n) \log{\alpha_q} + \sum
_{l
= 1}^{L} N_l(
\mathbf{w}_m) \log{\beta_l}.
\end{eqnarray*}

Also note that we control the number of errors made by a maximum a
posteriori clustering procedure only in the case where
$\Sig=\{(\mathit{Id},\mathit{Id})\}$, namely when there are no symmetries in the set
of matrices $\Pi_{\Q\calL}$. In the other cases, this procedure is
likely to select a configuration equivalent to the true one, but not
equal to it. We stress again the fact that the equivalence relation
is different from the label switching issue that can not be avoided
in finite mixture models. Moreover, exactly as for the label
switching issue, this phenomenon will not affect clustering
performance.
\end{remark}

\begin{pf*}{Proof of Theorem~\ref{thm:cv_posterior}}
We shall
exhibit the set $\Omega_1$ on which
Inequalities \eqref{eq:cv_post_ratio_1}
and \eqref{eq:cv_post_ratio_2} are satisfied using LBM notation, the
case of SBM easily follows.

First, note that we
have
\begin{eqnarray*}
&&\log\frac{p_{n,m}^{\theta}(s(\mathbf{Z}_n),t(\mathbf
{W}_{m}))}{p_{n,m}^{\theta}(\mathbf{z}_n,\mathbf{w}_m)}\\
&&\quad = \delta^{\bpi}\bigl(s(\mathbf{Z}_n),t(
\mathbf{W}_{m}),\mathbf {z}_n,\mathbf{w}_m
\bigr) +\sum_{i=1}^n \log \biggl(
\frac{\alpha_{s(Z_i)}}{\alpha_{z_i}} \biggr) +\sum_{j=1}^m
\log \biggl(\frac{\beta_{t(W_j)}}{\beta_{w_j}} \biggr).
\end{eqnarray*}
Thus, by letting $K = \log(\alpha_{\max}/\alpha_{\min}) \vee\log
(\beta_{\max}/\beta_{\min})$,
Inequalities \eqref{eq:cv_post_ratio_1}
and \eqref{eq:cv_post_ratio_2} are satisfied as soon as we
have
%
\begin{equation}
\label{eq:cv_post_ratio_delta} \bigl(c-2L_0\bigl\|\bpi-\boldsymbol{\pi}^{\star}
\bigr\|_\infty\bigr) (mr_1 + nr_2) \le
\delta^{\bpi}\bigl(s(\mathbf{Z}_n),t(\mathbf{W}_{m}),
\mathbf {z}_n,\mathbf{w}_m\bigr) \le C (mr_1
+ nr_2).
\end{equation}
Note that the latter inequality is defined on the set of equivalent
configurations $\tU$ and we can thus replace $(s(\mathbf{Z}_n),
t(\mathbf{W}_{m}))$ by
$(\mathbf{Z}_n, \mathbf{W}_{m})$. Let $(\mathbf{z}_n^{\star
},\mathbf{w}_m^{\star}) $ be a
fixed configuration in $\tU$, consider $(\mathbf{z}_n,\mathbf
{w}_m)\in\tU$. Whenever
$(\mathbf{z}_n,\mathbf{w}_m)\sim(\mathbf{z}_n^{\star},\mathbf
{w}_m^{\star})$, we have $r_1+r_2=0$ and the previous
inequality is automatically satisfied. Thus, we consider $(\mathbf
{z}_n,\mathbf{w}_m
)\in\tU$
such that
$(\mathbf{z}_n,\mathbf{w}_m) \neq(\mathbf{z}_n^{\star},\mathbf
{w}_m^{\star}) $ and let
$r_1:=\|\mathbf{z}_n^{\star}-\tilde{s}(\mathbf{z}_n)\|_0$ and
$r_2:=\|\mathbf{w}_m^{\star}-\tilde{t}(\mathbf{w}_m)\|_0$, where
$(\tilde{s},\tilde{t})\in
\Sig$ realizes the distance $d((\mathbf{z}_n^{\star},\mathbf
{w}_m^{\star}),(\mathbf{z}_n,\mathbf{w}_m))$. We consider
the event
\begin{eqnarray*}
A\bigl(\mathbf{z}_n^{\star},\mathbf{w}_m^{\star},
\mathbf{z}_n,\mathbf {w}_m\bigr) &=& \bigl\{
\delta^{\bpi}\bigl(\mathbf{z}_n^{\star},\mathbf
{w}_m^{\star},\mathbf{z}_n ,\mathbf{w}_m
\bigr) < \bigl(c-2L_0\bigl\|\bpi-\boldsymbol{\pi}^{\star}
\bigr\|_\infty\bigr) (mr_1 + nr_2) \bigr\}
\\
&&{}\cup \bigl\{ \delta^{\bpi}\bigl(\mathbf{z}_n^{\star},
\mathbf {w}_m^{\star},\mathbf{z}_n,
\mathbf{w}_m\bigr) > C (mr_1 + nr_2) \bigr\}
,
\end{eqnarray*}
where the constants $c,C>0$ have been previously introduced in
Proposition~\ref{prop:CondExp_Order}. We also assume that $\bpi$
satisfies $c-2L_0\|\bpi-\boldsymbol{\pi}^{\star}\|_\infty>0$.
According to this same
proposition, as soon as the configuration $(\mathbf{z}_n^{\star
},\mathbf{w}_m^{\star})$ is regular
in the sense that it belongs to the set $\tU^0$ defined through
Equation \eqref{eq:good_set} and following lines, we obtain that on
the set $\{(\mathbf{Z}_n,\mathbf{W}_{m}) = (\mathbf{z}_n^{\star
},\mathbf{w}_m^{\star})\}$, we have
\begin{eqnarray*}
2\bigl(c-2L_0\bigl\|\bpi-\boldsymbol{\pi}^{\star}
\bigr\|_\infty\bigr) (mr_1 +nr_2) \leq
\mathbb{E}_{\star}^{\mathbf{z}_n^{\star}\mathbf{w}_m^{\star}} \bigl( \delta^{\bpi}\bigl(
\mathbf{z}_n^{\star},\mathbf{w}_m^{\star
},
\mathbf{z}_n,\mathbf{w}_m\bigr) \bigr) \le
\frac{C} 2(mr_1+ nr_2) .
\end{eqnarray*}

We now control the probability of this event. Conditionally on $\{
(\mathbf{Z}_n
,\mathbf{W}_{m}) = (\mathbf{z}_n^{\star},\mathbf{w}_m^{\star})\}$,
the event
$A(\mathbf{z}_n^{\star},\mathbf{w}_m^{\star},\mathbf{z}_n,\mathbf
{w}_m)$ is included in the two-sided deviation of $
\delta^{\bpi}(\mathbf{z}_n^{\star},\mathbf{w}_m^{\star},\mathbf
{z}_n,\mathbf{w}_m) $ from its conditional
expectation $\Delta^{\bpi}(\mathbf{z}_n^{\star},\mathbf
{w}_m^{\star},\mathbf{z}_n,\mathbf{w}_m)$ at a distance at least
\begin{eqnarray*}
&&\min\biggl\{\bigl(c-2L_0\bigl\|\bpi-\boldsymbol{\pi}^{\star}
\bigr\|_\infty\bigr) (mr_1 + nr_2), \frac{C} 2
(mr_1 + nr_2)\biggr\}
\\
&&\quad =\bigl(c-2L_0\bigl\|\bpi-\boldsymbol{\pi}^{\star}
\bigr\|_\infty\bigr) (mr_1 + nr_2)
\ge(c-2L_0\eta) (mr_1 + nr_2).
\end{eqnarray*}
In other words,
\begin{eqnarray}
&&A\bigl(\mathbf{z}_n^{\star},\mathbf{w}_m^{\star},
\mathbf{z}_n,\mathbf {w}_m\bigr)\cap\bigl\{(
\mathbf{Z}_n,\mathbf{W}_{m}) = \bigl(\mathbf{z}_n^{\star
},
\mathbf{w}_m^{\star}\bigr)\bigr\}
\nonumber\\
&&\quad \subset\biggl( \bigl\{ \bigl(\delta^{\bpi} - \Delta^{\bpi}\bigr) \bigl(
\mathbf{z}_n^{\star
},\mathbf{w}_m^{\star},
\mathbf{z}_n,\mathbf{w}_m\bigr) < -\bigl(c-2L_0
\bigl\|\bpi-\boldsymbol{\pi}^{\star}\bigr\|_\infty\bigr) (mr_1+
nr_2) \bigr\}
\nonumber\\
&&\hphantom{\quad \subset\biggl(}\cup \biggl\{ \bigl(\delta^{\bpi} - \Delta^{\bpi}\bigr) \bigl(
\mathbf{z}_n^{\star},\mathbf{w}_m^{\star
},
\mathbf{z}_n,\mathbf{w}_m\bigr) > \frac{C} 2
(mr_1+ nr_2) \biggr\} \biggr)
\nonumber\\
&&\quad \subset \bigl\{ \bigl|\bigl( \delta^{\bpi} - \Delta^{\bpi}\bigr)
\bigl(\mathbf{z}_n^{\star},\mathbf{w}_m^{\star
},
\mathbf{z}_n,\mathbf{w}_m\bigr) \bigr| > (c-2L_0
\eta) (mr_1+ nr_2) \bigr\},\nonumber
\end{eqnarray}
where the last inclusion comes from $(c - 2L_0\eta) \leq C/2$.

Combining this sets' inclusions with Assumption~\ref{hyp:concentration}
yields
%
\begin{eqnarray}
\label{eq:concentration2} &&\mathbb{P}_{\star}\bigl(A\bigl(\mathbf{z}_n^{\star},
\mathbf{w}_m^{\star
},\mathbf{z}_n,
\mathbf{w}_m\bigr)\cap\bigl\{(\mathbf{Z}_n,
\mathbf{W}_{m}) = \bigl(\mathbf{z}_n^{\star},
\mathbf{w}_m^{\star}\bigr)\bigr\}\bigr) \nonumber\\
&&\quad \le
\mathbb{P}_{\star}\bigl((\mathbf{Z}_n,\mathbf{W}_{m})
= \bigl(\mathbf {z}_n^{\star},\mathbf{w}_m^{\star}
\bigr)\bigr)
\nonumber\\[-8pt]\\[-8pt]
&&\qquad {}\times\mathbb{P}_{\star}^{\mathbf{z}_n^{\star}\mathbf
{w}_m^{\star}} \bigl( \bigl| \bigl(
\delta^{\bpi} - \Delta^{\bpi}\bigr) \bigl(\mathbf{z}_n^{\star},
\mathbf {w}_m^{\star},\mathbf{z}_n,
\mathbf{w}_m\bigr) \bigr| > (c-2L_0\eta) (mr_1+
nr_2) \bigr)
\nonumber\\
&&\quad \le2\exp\bigl[- \psi^\star(c-2L_0\eta) (mr_1
+nr_2 ) \bigr] \bmu\bigl(\mathbf {z}_n^{\star},
\mathbf{w}_m^{\star}\bigr).\nonumber
\end{eqnarray}
We now consider the set $\Omega_1$ defined by
%
\begin{eqnarray}
\label{eq:Omega1} \Omega_1 &=&\Omega_0 \cap \biggl(\bigcap
_{(\mathbf{z}_n,\mathbf{w}_m)
\in\tU} \overline{A(\mathbf{Z}_n,
\mathbf{W}_{m},\mathbf{z}_n,\mathbf {w}_m)}
\biggr)
\nonumber\\[-8pt]\\[-8pt]
&=& \bigcup_{(\mathbf{z}_n^{\star},\mathbf{w}_m^{\star}) \in\tU^0} \bigcap
_{(\mathbf{z}_n,\mathbf{w}_m) \in\tU} \bigl( \overline{A\bigl(\mathbf{z}_n^{\star},
\mathbf{w}_m^{\star
},\mathbf{z}_n,
\mathbf{w}_m\bigr)} \cap\bigl\{(\mathbf{Z}_n,
\mathbf{W}_{m}) = \bigl(\mathbf{z}_n^{\star} ,
\mathbf{w}_m^{\star}\bigr)\bigr\} \bigr).\nonumber
\end{eqnarray}
On the set $\Omega_1$,
Inequality \eqref{eq:cv_post_ratio_delta} and thus
Inequalities \eqref{eq:cv_post_ratio_1} and
\eqref{eq:cv_post_ratio_2} are both
satisfied. We let
\[
\tU^{ \mathbf{z}_n^{\star}\mathbf{w}_m^{\star}} := \tU \smallsetminus\bigl\{ \bigl(\mathbf{z}_n^{\star},
\mathbf{w}_m^{\star}\bigr) \bigr\} =\tU\smallsetminus\bigl\{
\bigl(s\bigl(\mathbf{z}_n^{\star}\bigr),t\bigl(\mathbf
{w}_m^{\star}\bigr)\bigr) ;(s,t)\in\Sig\bigr\},
\]
be the set of all configurations but those which are equivalent to
$(\mathbf{z}_n^{\star},\mathbf{w}_m^{\star})$.
Since for any $(s,t)\in\Sig$, the
event $A(\mathbf{z}_n^{\star},\mathbf{w}_m^{\star},s(\mathbf
{z}_n^{\star}),t(\mathbf{w}_m^{\star}))$ has $\mathbb{P}_{\star
}$-probability zero,
we may write
\begin{eqnarray*}
\overline\Omega_1 = \overline\Omega_0 \cup \biggl(
\bigcup_{(\mathbf
{z}_n^{\star}
,\mathbf{w}_m^{\star}) \in
\tU^0} \bigcup_{(\mathbf{z}_n,\mathbf{w}_m) \in\tU^{\mathbf
{z}_n^{\star}\mathbf{w}_m^{\star}}}
A\bigl(\mathbf{z}_n^{\star},\mathbf {w}_m^{\star},
\mathbf{z}_n ,\mathbf{w}_m\bigr) \cap \bigl\{(
\mathbf{Z}_n,\mathbf{W}_{m}) = \bigl(\mathbf{z}_n^{\star},
\mathbf {w}_m^{\star}\bigr)\bigr\} \biggr).
\end{eqnarray*}
We now partition the set of configurations $(\mathbf{z}_n,\mathbf
{w}_m)\in\tU^{\mathbf{z}_n^{\star}
\mathbf{w}_m^{\star}}$ according to
the distance of each point $(\mathbf{z}_n,\mathbf{w}_m)$ to
$(\mathbf{z}_n^{\star},\mathbf{w}_m^{\star})$.
We write the following disjoint union
%
\begin{eqnarray}\label{eq:decomp_config}
\tU^{\mathbf{z}_n^{\star}\mathbf{w}_m^{\star}} &:=& \bigsqcup_{r_1+r_2=1}^{n+m}
\tU^{\mathbf{z}_n^{\star}\mathbf
{w}_m^{\star}}(r_1,r_2)
\nonumber
\\
&:=& \bigsqcup_{r_1+r_2=1}^{n+m} \bigl\{(
\mathbf{z}_n,\mathbf{w}_m)\in \tU^{\mathbf{z}_n^{\star}\mathbf{w}_m^{\star}} ;d\bigl(
\bigl(\mathbf{z}_n^{\star},\mathbf{w}_m^{\star}
\bigr);\nonumber\\[-8pt]\\[-8pt]
&&\hphantom{\bigsqcup_{r_1+r_2=1}^{n+m} \bigl\{} (\mathbf {z}_n,\mathbf{w}_m)\bigr)=\bigl\|
\mathbf{z}_n^{\star}-s(\mathbf{z}_n)
\bigr\|_0
+\bigl\|\mathbf{w}_m^{\star}-t(\mathbf{w}_m)
\bigr\|_0 \mbox{ and}\nonumber\\
&&\hphantom{\bigsqcup_{r_1+r_2=1}^{n+m} \bigl\{} \bigl\|\mathbf{z}_n^{\star}-s(
\mathbf{z}_n)\bigr\|_0 =r_1, \bigl\|
\mathbf{w}_m^{\star}-t(\mathbf{w}_m)
\bigr\|_0 =r_2\bigr\}. \nonumber
\end{eqnarray}
Note that the above decomposition is not unique. Indeed, we may have
that the distance $d((\mathbf{z}_n^{\star},\mathbf{w}_m^{\star
});(\mathbf{z}_n,\mathbf{w}_m))$ $=r_1+r_2=r_1'+r_2'$
but $r_1\neq r_1'$ and
$r_2\neq r_2'$. In such a case, we make an arbitrary choice between
the couples $(r_1,r_2)$ and $(r_1',r_2')$ to represent the distance
from $(\mathbf{z}_n,\mathbf{w}_m)$ to $(\mathbf{z}_n^{\star
},\mathbf{w}_m^{\star})$.
This decomposition leads to
\begin{eqnarray*}
&&\mathbb{P}_{\star}(\overline\Omega_1) \le
\mathbb{P}_{\star
}(\overline\Omega_0) + 2\sum
_{(\mathbf{z}_n^{\star},\mathbf{w}_m^{\star})\in\tU^0} \bmu\bigl(\mathbf{z}_n^{\star},
\mathbf{w}_m^{\star}\bigr)
\\
&&\hphantom{\mathbb{P}_{\star}(\overline\Omega_1) \le
\mathbb{P}_{\star
}(\overline\Omega_0) + 2\sum
_{(\mathbf{z}_n^{\star},\mathbf{w}_m^{\star})\in\tU^0}}{}\times\sum_{r_1+r_2=1}^{n+m} \bigl|
\tU^{\mathbf{z}_n^{\star}\mathbf{w}_m^{\star}}(r_1,r_2)\bigr| \exp \bigl[-
\psi^\star(c-2L_0\eta) (mr_1+nr_2)
\bigr] .
\end{eqnarray*}
Now, we use the bound
%
\begin{equation}
\label{eq:cardinal} \bigl|\tU^{\mathbf{z}_n^{\star}\mathbf{w}_m^{\star}}(r_1,r_2)\bigr|\le|\Sig |
{n \choose r_1} {m \choose r_
2} ,
\end{equation}
which leads to
\begin{eqnarray*}
\mathbb{P}_{\star}(\overline\Omega_1) &\le&
\mathbb{P}_{\star
}(\overline\Omega_0) + 2 \sum
_{r_1+r_2=1}^{n+m} |\Sig| {n \choose r_1} {m
\choose r_ 2} \exp\bigl[-\psi ^\star(c-2L_0\eta)
(mr_1+nr_2) \bigr]
\\
& \le&\mathbb{P}_{\star}(\overline\Omega_0) + 2 |\Sig| \bigl[
\bigl\{1+\exp\bigl[-m \psi^\star(c-2L_0\eta) \bigr]\bigr
\}^{n} \bigl\{1+\exp\bigl[-n\psi^\star(c-2L_0\eta)
\bigr]\bigr\}^{m} - 1 \bigr] .
\end{eqnarray*}
We now rely on the following bound, valid for any $u,v>0$,
%
\begin{equation}
\label{eq:(1+u)^n} (1+u)^{n} \times(1+v )^{m} - 1 \le(nu+mv)
\exp(nu+mv).
\end{equation}
Combining the latter with the control of the probability of $\overline
\Omega_0$ given in Proposition~\ref{prop:CondExp_Order}, we obtain
\begin{eqnarray*}
\mathbb{P}_{\star}(\overline\Omega_1) \le2QL \exp\bigl(-(n
\wedge m) \mu_{\min}^2/2\bigr) + 2 |\Sig| d_{n,m}
\exp(d_{n,m}),
\end{eqnarray*}
where $d_{n,m} = [n\exp\{-\psi^\star(c-2L_0\eta) m \} + m\exp\{
-\psi^\star(c-2L_0\eta) n \}]$.

Note that as soon as $(m_n)_{n\ge1} $ is a sequence such that $m_n\to
+\infty$ and $(\log
n)/m_n \to0$, we
obtain that for any constant $a>0$, the sequence $u_n=n\exp(-am_n)$ is
negligible with respect to $n^{-1-s}$, for any $s>0$, and thus $\sum_n
u_n <+\infty$. In particular, the sequence
\[
\eps_{n,m} :=2QL \exp\bigl[-(n \wedge m) \mu_{\min}^2/2
\bigr] + 2 |\Sig | d_{n,m} \exp(d_{n,m})
\]
satisfies $\sum_{n} \eps_{n,m_n} <+\infty$.
As for SBM, it is easy to see that this expression reduces to
\[
\eps_{n,n} :=2Q \exp\bigl[-n \alpha_{\min}^2/2
\bigr] + 2 |\Sig| d_n \exp(d_n)
\]
with $d_n= n\exp\{-2\psi^\star(c-2L_0\eta) n\} $ and which also
satisfies $\sum_{n} \eps_{n,n} <+\infty$. This concludes the proof.
\end{pf*}
\begin{pf*}{Proof of Corollaries \ref{cor:1}, \ref{cor:2} and \ref{cor:3}}
The proof of these three corollaries relies on the same scheme that we
shall now present in LBM notation. The proof is easily generalised
to SBM.
First, note that $\Omega_1 = \bigcup_{(\mathbf{z}_n^{\star},\mathbf
{w}_m^{\star}) \in
\U^0} (\Omega_1 \cap
\{(\mathbf{Z}_n,\mathbf{W}_{m}) = (\mathbf{z}_n^{\star},\mathbf
{w}_m^{\star}) \})$.
Let us fix some configuration $(\mathbf{z}_n^{\star},\mathbf
{w}_m^{\star}) $ in $\U^0$. On the set
$\Omega_1 \cap\{(\mathbf{Z}_n,\mathbf{W}_{m}) = (\mathbf
{z}_n^{\star},\mathbf{w}_m^{\star}) \}$, we have
\begin{eqnarray*}
1 - p_{n,m}^{\theta} \bigl( \bigl\{ (\mathbf{Z}_n,
\mathbf{W}_{m}) \bigr\} \bigr) &\le& \frac{1 - p_{n,m}^{\theta} ( \{ (\mathbf{Z}_n,\mathbf{W}_{m})
\}
) }{p_{n,m}^{\theta} ( \{ (\mathbf{Z}_n,\mathbf{W}_{m})
\}  )}
\\
&=& \mathop{\mathop{\sum}_{(\mathbf{z}_n,\mathbf{w}_m)\in\U}}_{ (\mathbf
{z}_n,\mathbf{w}_m)\nsim(\mathbf{z}_n^{\star}
, \mathbf{w}_m^{\star})} \exp \biggl(-
\log\frac{p_{n,m}^{\theta} ( \{ (\mathbf
{z}_n^{\star},\mathbf{w}_m^{\star}) \}  ) }{
p_{n,m}^{\theta}(\mathbf{z}_n,\mathbf{w}_m)} \biggr) ,
\end{eqnarray*}
where we abbreviate to $ \{ (\mathbf{Z}_n,\mathbf{W}_{m}) \}$ and $\{
(\mathbf{z}_n^{\star},\mathbf{w}_m^{\star}) \} $ the
whole sets of configurations $ \{ (\mathbf{z}_n, \mathbf{w}_m)
\sim(\mathbf{Z}_n,\mathbf{W}_{m}) \}$
and $ \{ (\mathbf{z}_n, \mathbf{w}_m) \sim(\mathbf{z}_n^{\star
},\mathbf{w}_m^{\star}) \}$, respectively.
Let $(\mathbf{z}_n, \mathbf{w}_m) \nsim(\mathbf{z}_n^{\star},
\mathbf{w}_m^{\star})$. There exists $(s, t) \in
\Sig$ such that $\| \mathbf{z}_n- s(\mathbf{z}_n^{\star}) \|_0 =
r_1$ and $\| \mathbf{w}_m- t(\mathbf{w}_m^{\star})
\|_0 = r_2$. Using Inequality \eqref{eq:cv_post_ratio_1} and $\|\bpi
-\bpi^\star\|_\infty
\le\eta$, we get
\begin{eqnarray*}
\log\frac{p_{n,m}^{\theta} ( \{ (\mathbf{z}_n^{\star},\mathbf
{w}_m^{\star}) \}
) }{ p_{n,m}^{\theta}(\mathbf{z}_n,\mathbf{w}_m)} \ge\log \frac{p_{n,m}^{\theta}(s(\mathbf{z}_n^{\star}),
t(\mathbf{w}_m^{\star})) }{ p_{n,m}^{\theta}(\mathbf{z}_n,\mathbf
{w}_m) } \ge(c-2L_0\eta)
(mr_1+nr_2) + K(r_1 + r_2)
\end{eqnarray*}
and therefore
%
\begin{equation}
\label{eq:borne_inf_p} 1- p_{n,m}^{\theta} \bigl( \bigl\{ (
\mathbf{Z}_n,\mathbf{W}_{m}) \bigr\} \bigr) \le \mathop{
\mathop{\sum}_{(\mathbf{z}_n,\mathbf{w}_m)\in\U}}_{ (\mathbf
{z}_n,\mathbf{w}_m)\nsim(\mathbf{z}_n^{\star},
\mathbf{w}_m^{\star})} \exp\bigl[-
(c-2L_0\eta) (mr_1+nr_2) +
K(r_1 + r_2)\bigr] .
\end{equation}
When $\Sig=\{(\mathit{Id},\mathit{Id})\}$, the set $ \{ (\mathbf{z}_n, \mathbf{w}_m)
\sim(\mathbf{Z}_n,\mathbf{W}_{m})
\} $ reduces to a singleton and the previous bound becomes
\begin{eqnarray*}
1 - p_{n,m}^{\theta}(\mathbf{Z}_n,
\mathbf{W}_{m}) \le\mathop{\mathop{\sum} _{(\mathbf{z}_n,\mathbf{w}_m)\in\U}}_{ (\mathbf{z}_n
,\mathbf{w}_m)\nsim(\mathbf{z}_n^{\star}, \mathbf{w}_m^{\star
})}
\exp\bigl[- (c-2L_0\eta) (mr_1+nr_2) +
K(r_1+r_2)\bigr] .
\end{eqnarray*}
Using the decomposition \eqref{eq:decomp_config} on the set $\tU
^{\mathbf{z}_n^{\star}\mathbf{w}_m^{\star}}$ and the bound \eqref
{eq:cardinal} on the cardinality of
each $\tU^{\mathbf{z}_n^{\star},\mathbf{w}_m^{\star}}(r_1,r_2)$,
we get
\begin{eqnarray*}
1 - p_{n,m}^{\theta}(\mathbf{Z}_n,
\mathbf{W}_{m}) &\le&\sum_{r_1
+r_2= 1}^{n+m}
{n \choose r_1} {m \choose r_ 2} \exp\bigl[-
(c-2L_0\eta) (mr_1+nr_2)+
K(r_1+r_2) \bigr]\nonumber
\\
&=& \bigl\{\bigl(1+\exp(-mc_1+K)\bigr)^n \bigl(1+
\exp(-nc_1+K )\bigr)^m - 1 \bigr\},\nonumber
\end{eqnarray*}
where $c_1=c-2L_0\eta$.
Using again Inequality \eqref{eq:(1+u)^n}, we obtain
\begin{eqnarray*}
1 - p_{n,m}^{\theta}(\mathbf{Z}_n,
\mathbf{W}_{m}) \le a_{n,m} \exp (a_{n,m}) ,
\end{eqnarray*}
where $ a_{n,m} = (n\mathrm{e}^{-(c-2L_0\eta)m+K } +
m\mathrm{e}^{-(c-2L_0\eta)n+K })$. In SBM, this quantity becomes $a_{n,n} =
n\mathrm{e}^{-2(c-2L_0\eta)n+K } $.

The case where $K=0$ is handled similarly and gives
\begin{eqnarray*}
1 - p_{n,m}^{\theta} \bigl( \bigl\{ (\mathbf{Z}_n,
\mathbf{W}_{m}) \bigr\} \bigr) \le|\Sig|a_{n,m} \exp
(a_{n,m}) ,
\end{eqnarray*}
with the same definition of $a_{n,m}$, replacing $K$ with $0$.

Moreover when $K = 0$, we have $\alpha_1 = \cdots= \alpha_Q$ and
$\beta_1 = \cdots= \beta_L$ and it easy to check that
\begin{eqnarray*}
p_{n,m}^{\theta}(\mathbf{Z}_n,\mathbf{W}_{m})
= p_{n,m}^{\theta
}\bigl(s(\mathbf{Z}_n),t(
\mathbf{W}_{m})\bigr)
\end{eqnarray*}
for all $(s,t) \in\Sig$.

Now, in the general case, we come back to \eqref{eq:borne_inf_p}.
Using the decomposition \eqref{eq:decomp_config} on the set $\tU
^{\mathbf{z}_n^{\star}\mathbf{w}_m^{\star}}$ and the bound \eqref
{eq:cardinal} on the cardinality of
each $\tU^{\mathbf{z}_n^{\star},\mathbf{w}_m^{\star}}(r_1,r_2)$,
we get
\begin{eqnarray*}
1 - p_{n,m}^{\theta} \bigl( \bigl\{ (\mathbf{Z}_n,
\mathbf{W}_{m}) \bigr\} \bigr) &\le&\sum_{r_1 +r_2=
1}^{n+m}
|\Sig| {n \choose r_1} {m \choose r_ 2} \exp\bigl[-
(c-2L_0\eta) (mr_1+nr_2) +
K(r_1+r_2)\bigr]
\\
& \le&|\Sig| \bigl\{\bigl(1+\exp(-mc_1 + K)\bigr)^n
\bigl(1+\exp(-nc_1 + K)\bigr)^m - 1 \bigr\},
\end{eqnarray*}
where $c_1=c-2L_0\eta$. Using again Inequality \eqref{eq:(1+u)^n},
we obtain
\begin{eqnarray*}
1 - p_{n,m}^{\theta} \bigl( \bigl\{ (\mathbf{z}_n,
\mathbf{w}_m) \sim (\mathbf{Z}_n,\mathbf{W}_{m})
\bigr\} \bigr) \le |\Sig| a_{n, m}\exp(a_{n,m}) ,
\end{eqnarray*}
with same definition of $a_{n,m}$ as previously.

We now provide an upper bound for the posterior probability of the
class $\{ (\mathbf{z}_n, \mathbf{w}_m) \sim(\mathbf{Z}_n,\mathbf
{W}_{m}) \} $, valid on the set
$\Omega_1$. Let us fix some configuration $(\mathbf{z}_n^{\star
},\mathbf{w}_m^{\star}) $ in $\U^0$.
On the set $\Omega_1 \cap\{(\mathbf{Z}_n,\mathbf{W}_{m}) = (\mathbf
{z}_n^{\star},\mathbf{w}_m^{\star}) \}$, we have
\begin{eqnarray*}
\frac{1} {p_{n,m}^{\theta} ( \{ (\mathbf{Z}_n,\mathbf{W}_{m}) \}
)} =1+\sum_{(\mathbf{z}_n,\mathbf{w}_m)\nsim(\mathbf{Z}_n,\mathbf
{W}_{m})}\exp \biggl(- \log
\frac{p_{n,m}^{\theta} ( \{ (\mathbf{Z}_n,\mathbf{W}_{m}) \}
) }{ p_{n,m}^{\theta}(\mathbf{z}_n,\mathbf{w}_m)} \biggr)
\end{eqnarray*}
and relying on Inequality \eqref{eq:cv_post_ratio_2}, we get
\begin{eqnarray*}
p_{n,m}^{\theta} \bigl( \bigl\{ \bigl(\mathbf{z}_n^{\star},
\mathbf {w}_m^{\star}\bigr) \bigr\} \bigr) &\le& \biggl\{1+\sum
_{(\mathbf{z}_n,\mathbf{w}_m)\nsim(\mathbf
{z}_n^{\star},\mathbf{w}_m^{\star})}\exp \biggl(- \log \frac{p_{n,m}^{\theta} ( \{ (\mathbf{z}_n^{\star},\mathbf
{w}_m^{\star}) \}
) }{ p_{n,m}^{\theta}(\mathbf{z}_n,\mathbf{w}_m)} \biggr)
\biggr\}^{-1}
\\
& \le& \biggl\{1+\sum_{(\mathbf{z}_n,\mathbf{w}_m)\nsim(\mathbf
{z}_n^{\star},\mathbf{w}_m^{\star})}\exp \bigl( -C
(mr_1 + nr_2) - K (r_1 + r_2)
\bigr) \biggr\}^{-1}.
\end{eqnarray*}
Following the same lines, we obtain the desired upper-bounds.
\end{pf*}


\section{Examples of application}\label{sec:applis}

The goal of this section is to derive the results of
Theorem~\ref{thm:cv_posterior} and following corollaries in many
different setups. The key ingredient for that lies in establishing
the concentration of the ratio $\delta^{\bpi}$ around its conditional
expectation $\Delta^{\bpi}$ (namely
Assumption~\ref{hyp:concentration}). As mentioned in
Remarks \ref{rem:concentration} and \ref{rem:exponential_families}, it
is valid for many exponential families. We will first present the
general proof for exponential families and then state the results for
common exponential families.

\subsection{Scheme of proof of concentration inequalities}
\label{sec:scheme}

One of the main issues for Theorem~\ref{thm:cv_posterior} to be valid
is the existence of a concentration of the ratio $\delta^{\bpi}$
around its conditional expectation $\Delta^{\bpi}$, namely
Assumption~\ref{hyp:concentration}. This section presents the general
methodology that will be employed.

The scheme of proof is as follows. Relying on the notation of
Assumption~\ref{hyp:concentration} and using \eqref{eq:esp_delta}, we
write
\begin{eqnarray*}
&&\delta^{\bpi} \bigl(\mathbf{z}_n^{\star},
\mathbf{w}_m^{\star},\mathbf {z}_n,
\mathbf{w}_m\bigr) - \Delta^{\bpi} \bigl(\mathbf{z}_n^{\star},
\mathbf{w}_m^{\star},\mathbf{z}_n,\mathbf
{w}_m\bigr)
\\
&&\quad = \sum_{(i,j) \in\I} \log \biggl( \frac{f(X_{ij};\pi_{z^\star_i
w^\star_j})}{f(X_{ij};\pi_{z_i w_j})} \biggr)
- \mathbb{E}_{\theta}^{\mathbf{z}_n^{\star}\mathbf{w}_m^{\star}} \log \biggl( \frac{f(X_{ij};\pi_{z^\star_i
w^\star_j})}{f(X_{ij};\pi_{z_i w_j})}
\biggr) := \sum_{(i,j)\in\I}Y_{ij}.
\end{eqnarray*}
Conditional on $(\mathbf{Z}_n,\mathbf{W}_{m})=(\mathbf{z}_n^{\star
},\mathbf{w}_m^{\star})$, the random variables $Y_{ij}$
are independent and centered. There are exactly
$D:=\diff(\mathbf{z}_n^{\star},\mathbf{w}_m^{\star},\mathbf
{z}_n,\mathbf{w}_m)$ such non-null variables and since
$D \le mr_1+nr_2-r_1r_2 \le mr_1+nr_2$, we may write
%
\begin{equation}
\label{eq:etape_conc} \mathbb{P}_{\star}^{\mathbf{z}_n^{\star}\mathbf{w}_m^{\star}} \bigl( \bigl\llvert
\bigl(\delta^{\bpi}- \Delta^{\bpi}\bigr) \bigl(
\mathbf{z}_n^{\star},\mathbf{w}_m^{\star},
\mathbf {z}_n,\mathbf{w}_m\bigr) \bigr\rrvert \ge
\eps(mr_1+nr_2) \bigr) \le\mathbb{P}_{\star}^{\mathbf{z}_n^{\star}\mathbf{w}_m^{\star}}
\biggl( \biggl| \sum_{(i,j)\in\I}Y_{ij} \biggr| \ge\eps D
\biggr) .
\end{equation}
Thus, the problem boils down to establishing a concentration
inequality for the sum $\sum Y_{ij}$ composed of $D$ conditionally
independent and centered random variables. As soon as we have the
existence of a positive function $\psi^\star_{\max}$ such that for any
$\epsilon>0$,
%
\begin{equation}
\label{eq:key_conc} \mathbb{P}_{\star}^{\mathbf{z}_n^{\star}\mathbf{w}_m^{\star}} \biggl(\biggl | \sum
_{(i,j)\in\I}Y_{ij} \biggr| \ge \eps D \biggr) \le2\exp\bigl\{-
\psi^\star_{\max}(\eps) D\bigr\},
\end{equation}
we can combine Lemma~\ref{lem:Bound_Number} and
bound \eqref{eq:etape_conc} to obtain
\begin{eqnarray*}
\mathbb{P}_{\star}^{\mathbf{z}_n^{\star}\mathbf{w}_m^{\star}} \biggl( \biggl| \sum
_{(i,j)\in\I}Y_{ij} \biggr| \ge \eps (mr_1+nr_2)
\biggr) &\le&2\exp \bigl\{-\psi^\star_{\max}(\eps)
\mu^2_{\min}(mr_1+nr_2)/8 \bigr\}
\\
&:=&2\exp \bigl\{-\psi^\star(\eps) (mr_1+nr_2)
\bigr\} ,
\end{eqnarray*}
with $\psi^\star(\cdot)=\psi^\star_{\max}(\cdot)\mu_{\min
}^2/8$. Note
that Inequality \eqref{eq:key_conc} is often obtained through a
Cramer--Chernoff bound in the following way. We let
$\psi_{ij}(\lambda):=\log\mathbb{E}_{\star}^{\mathbf{z}_n^{\star
}\mathbf{w}_m^{\star}}(\exp(\lambda
Y_{ij}))$, for
any $\lambda>0$ such that this quantity is finite, let us say $\lambda
\in I\subset\mathbb{R}$. Using a Cramer--Chernoff bound, we get for
any $x>0$,
\begin{eqnarray*}
\mathbb{P}_{\star}^{\mathbf{z}_n^{\star}\mathbf{w}_m^{\star}} \bigl( |Y_{ij} | \ge x \bigr) \le 2
\exp \Bigl\{ -\sup_{\lambda\in I} \bigl(\lambda x -\psi_{ij}(
\lambda )\bigr) \Bigr\}.
\end{eqnarray*}
As soon as we can uniformly bound this quantity (uniformly with respect
to $i,j$ and also
underlying $\bpi$), namely if we can
write
\begin{eqnarray*}
\mathbb{P}_{\star}^{\mathbf{z}_n^{\star}\mathbf{w}_m^{\star}} \bigl( |Y_{ij} | \ge x \bigr) \le 2
\exp \Bigl\{ -\sup_{\lambda\in I} \bigl(\lambda x -\psi_{\max}(
\lambda )\bigr) \Bigr\},
\end{eqnarray*}
with $\psi_{\max}:= \sup_{\bpi\in\Pi_{\Q\calL}} \max_{(i,j)\in
\I}\psi_{ij}$, the conditional
independence of the $Y_{ij}$'s gives that for any $\epsilon>0$, and
any $\lambda>0$,
\[
\mathbb{P}_{\star}^{\mathbf{z}_n^{\star}\mathbf{w}_m^{\star}} \biggl( \biggl| \sum
_{(i,j)\in\I}Y_{ij} \biggr| \ge \eps D \biggr) \le2\exp\bigl\{ -
\bigl(\lambda\epsilon D -D\psi_{\max}(\lambda)\bigr)\bigr\},
\]
leading to
\[
\mathbb{P}_{\star}^{\mathbf{z}_n^{\star}\mathbf{w}_m^{\star}} \biggl( \biggl| \sum
_{(i,j)\in\I}Y_{ij} \biggr| \ge \eps D \biggr) \le2\exp\Bigl\{
-D\sup_{\lambda\in I} \bigl(\lambda\epsilon -\psi_{\max}(
\lambda)\bigr)\Bigr\} \le2\exp \bigl\{-D \psi^\star_{\max}(
\epsilon) \bigr\},
\]
where $\psi^\star_{\max} (\epsilon): = \sup_{\lambda\in I}
(\lambda
\epsilon-\psi_{\max}(\lambda))$. Note that since $\psi_{ij}(0)=0$, we
have $\psi_{\max}(0)=0$ and $\psi^\star_{\max}$ is non-negative.

\subsection{Examples from exponential families}
\label{sec:comm-expon-famil}

We state here the rate functions $\psi^\star$ and validity
assumptions of our main result under several models for the
observations $X_{ij}$, all included in the
exponential family framework.

\subsubsection*{Binary model} Let $X_{ij} \in\{0,1\}$ and
$f(\cdot;\pi)$ a Bernoulli distribution with parameter $\pi$.
Assumptions \ref{hyp:concentration} to \ref{hyp:Lipschitz} are
satisfied if the paramater set is bounded away from 0 and 1, namely
$\Pi\subset[a, 1-a]$ for some $a \in(0, 1/2)$. The corresponding
rate function, given by Hoeffding's inequality is
%
\begin{equation}
\label{eq:binary_rate_function_hoeffding} \psi^\star(x)= x^2 \mu_{\min}^2
/\bigl\{16\bigl[\log(1-a)-\log a\bigr]^2\bigr\} .
\end{equation}
In the interesting special case where $\Pi\subset\xi[a, 1 -a]$ with
$\xi> 0$ small, Bernstein's inequality gives the sharper rate function
\begin{eqnarray*}
\psi^\star(x)= x^2\mu_{\min}^2 /\bigl
\{64\xi\bigl[\log(1-a)-\log a\bigr]^2 + 32x\bigl[\log(1-a) - \log a
\bigr]/3\bigr\}
\end{eqnarray*}
which gives the following rate function for deviations of order $\xi$:
%
\begin{eqnarray}
\label{eq:binary_rate_function_bernstein} \psi^\star(\xi x) & = & \xi\breve{\psi}^\star(x)\qquad
\mbox{where}
\nonumber\\[-8pt]\\[-8pt]
\breve{\psi}^\star(x) & = & x^2\mu_{\min}^2
/ \bigl\{64\bigl[\log(1-a)-\log a\bigr]^2 + 32x\bigl[\log(1-a) - \log
a\bigr]/3\bigr\}.
\nonumber
\end{eqnarray}
This small deviation rate function is useful for sparse asymptotics,
as studied in Section~\ref{sec:pi_vers0}.

\subsubsection*{Binomial model}\label{binary} Let $X_{ij} \in
\{0,\dots,p\}$ and $f(\cdot,\pi)$ a binomial distribution
$\mathcal{B}(p, \pi)$. Assumptions \ref{hyp:concentration}
to \ref{hyp:Lipschitz} are satisfied if the paramater set is bounded
away from 0 and 1, namely $\Pi\subset[a, 1-a]$ for some $a \in(0,
1/2)$. The corresponding rate function, given by Hoeffding's
inequality is
\begin{eqnarray*}
\psi^\star(x)= x^2 \mu_{\min}^2 /\bigl
\{16p^2\bigl[\log(1-a)-\log a\bigr]^2\bigr\}.
\end{eqnarray*}

\subsubsection*{Multinomial model} Let $X_{ij}$ be
discrete with $p$ levels labelled $1$ to $p$, parameter $\pi=
(\pi(1), \dots, \pi(p))$ and $f(k, \pi) = \pi(k)$.
Assumptions \ref{hyp:concentration} to \ref{hyp:Lipschitz} are
satisfied if the paramater set for $(\pi(k))_{1 \leq k \leq p}$ is
bounded away from 0 and 1, namely $\Pi\subset[a, 1-a]^p$ for some $a
\in(0, 1/2)$. The corresponding rate function, given by Hoeffding's
inequality is
\begin{eqnarray*}
\psi^\star(x)= x^2 \mu_{\min}^2 /\bigl
\{8p\bigl[\log(1-a)-\log a\bigr]^2\bigr\}.
\end{eqnarray*}

\subsubsection*{Poisson model} Let $X_{ij} \in\mathbb{N}$
and $f(\cdot, \pi)$ is a Poisson distribution with parameter $\pi$.
Assumptions \ref{hyp:concentration} to \ref{hyp:Lipschitz} are
satisfied if the paramater set is bounded away from 0 and $+\infty$,
namely $\Pi\subset[\pi_{\min}, \pi_{\max}] \subset(0, +\infty)$.
The corresponding rate function, given by a Cramer--Chernoff bound for
Poisson variable (see, for instance, \cite{Massart_07}) is
\begin{eqnarray*}
\psi^\star(x) = \frac{1} 8 \mu_{\min}^2
\pi_{\max} h \biggl( \frac
{x}{\pi_{\max}
\log(\pi_{\max} / \pi_{\min})} \biggr) ,
\end{eqnarray*}
where $\forall u \geq-1 , h(u) = (1+u)\log(1+u) - u $.

\subsubsection*{Gaussian location model} We are interested here in
Gaussian observations in the homoscedastic case. We assume that
$X_{ij} \in\mathbb{R}$ and $f(\cdot, \pi)$ is a Gaussian distribution
with mean value $\pi$ and fixed variance $\sigma^2$, namely $f(x,
\pi_{ij}) = c \exp\{ -(x - \pi_{ij})^2 / (2\sigma^2) \}$ where $c$ is
a normalizing constant. Assumptions \ref{hyp:concentration}
to \ref{hyp:Lipschitz} are satisfied if the paramater set is bounded
away from $-\infty$ and $+\infty$, namely $\Pi\subset[\pi_{\min},
\pi_{\max}] \subset\mathbb{R}$. The corresponding rate function,
given by a Cramer--Chernoff bound for Gaussian variables is
\begin{eqnarray*}
\psi^\star(x)= \frac{\mu_{\min}^2 \sigma^2 x^2}{16(\pi_{\max} -
\pi_{\min})^2} .
\end{eqnarray*}

\subsubsection*{Gaussian scale model} We are interested here in Gaussian
observations with fixed mean and different variances. We assume that
$X_{ij} \in\mathbb{R}$ and $f(\cdot, \pi)$ is a Gaussian distribution
with fixed mean value $m$ and variance $\pi\in(0, +\infty)$, namely
$f(x, \pi_{ij}) = c(\pi_{ij})^{-1/2} \exp\{ -(x - m)^2 / (2\pi_{ij}^2)
\}$ where $c$ is a normalizing constant.
Assumptions \ref{hyp:concentration} to \ref{hyp:Lipschitz} are
satisfied if the paramater set is bounded away from $0$ and $+\infty$,
namely $\Pi\subset[\pi_{\min}, \pi_{\max}] \subset(0, +\infty
)$. The
corresponding rate function, given by a Cramer--Chernoff bound for
$\chi^2(1)$ random variables is
\begin{eqnarray*}
\psi^\star(x)= \frac{\mu_{\min}^2\sigma^2x}{8(\pi_{\max} -
\pi_{\min})} + \frac{\mu_{\min}^2}{16} \log \biggl\{ 1
+ \frac{2\pi_{\min}x}{\pi_{\max} - \pi_{\min}} \biggr\} .
\end{eqnarray*}

\subsection{Zero-inflated distributions}\label{sec:weighted}
Here, we assume that $X_{ij}$ follows a mixture of a Dirac mass at
zero with another distribution (on $\mathbb{R}$ for instance). This
situation is particularly relevant for modeling sparse matrices
\cite{Ambroise_Matias}. In this context, the former parameter $\pi$
becomes now $(\pi,\gamma)\in(0,1)\times\Gamma$ and we let
\begin{eqnarray*}
f(\cdot;\pi,\gamma)= \pi\tilde{f}(\cdot;\gamma) +(1-\pi)\delta _0(
\cdot),
\end{eqnarray*}
where $\delta_0$ is the Dirac mass at $0$. The set of parameter
matrices $ (\bpi,\boldsymbol{\gamma}):=((\pi_{ql}),
( \gamma_{ql}))_{q \in\Q, l \in\calL}$ is
denoted by $\Pi_{\Q\calL}\times\Gamma_{\Q\calL}$.
For identifiability
reasons, we also constrain the parametric family
$\{ \tilde{f}(\cdot;\gamma) ;\gamma\in\Gamma\}$ such that any
distribution in this set admits a continuous cumulative distribution
function (c.d.f.) at zero. Moreover, we shall assume that the
distributions $\{ \tilde{f}(\cdot;\gamma); \gamma\in\Gamma\}$
satisfy Assumption~\ref{hyp:Bound_KL}.

For instance, $\tilde{f}(\cdot;\gamma)$ may be absolutely continuous
with respect to the Lebesgue measure. Another interesting case
consists in considering the density (with respect to the counting
measure) of the Poisson distribution, with parameter $\gamma$, but
truncated at zero. Namely, for any $k\ge1$, we let $ \tilde
f(k ;\gamma)=\gamma^k/(k!) (\mathrm{e}^\gamma-1)^{-1}$. This leads to
zero-inflated Poisson models and more generally, one could consider
other zero-inflated counts models.

In the following, we will assume that the parameter set $\Pi$ is
included in $ [a,1-a]$ for some $a\in(0,1/2)$ and that the family
$\{\tilde{f}(\cdot;\gamma); \gamma\in\Gamma\}$ satisfies a
concentration property on its likelihood ratio statistics as follows.

\begin{assumption}
\label{hyp:concentration_tildef}
Fix $(\mathbf{z}_n^{\star},\mathbf{w}_m^{\star}) \in\tU_0$ and $
(\mathbf{z}_n,\mathbf{w}_m)$ in $\tU$ with
$(\mathbf{z}_n^{\star},\mathbf{w}_m^{\star}) \neq(\mathbf
{z}_n,\mathbf{w}_m)$. Let $\tilde Y_{ij}=\log[\tilde
f(X_{ij} ; \gamma_{z^\star_iw^\star_j})/\tilde f
(X_{ij};\gamma_{z_iw_j})] + c$, where $c$ is a centering constant.
There exists a positive function $\tilde
{\psi}^\star_{\max} \dvtx (0,+\infty)\to(0,+\infty]$ such that for any
$x>0$, for
any $(i,j)\in\I$, and any $\boldsymbol{ \gamma} \in\Gamma_{\Q
\calL}$,
\begin{eqnarray*}
\mathbb{P}_{\star}^{\mathbf{z}_n^{\star}\mathbf{w}_m^{\star}}\bigl(| \tilde Y_{ij}|\ge x
|X_{ij}\neq0\bigr) \le2 \exp\Bigl\{-\sup_{\lambda\in I} \bigl(
\lambda x-\tilde{\psi}_{\max
}(\lambda)\bigr)\Bigr\} := 2\exp\bigl(-
\tilde{\psi}^\star_{\max} (x)\bigr),
\end{eqnarray*}
where $\tilde{\psi}_{\max}(\lambda)= \sup_{\boldsymbol{\gamma}
\in\Gamma_{\Q\calL}} \max_{(i,j)\in\I}\log
\mathbb{E}_{\star}^{\mathbf{z}_n^{\star}\mathbf{w}_m^{\star}}(
\exp(\lambda\tilde{Y}_{ij}) |X_{ij}\neq0)$
exists for
any $\lambda\in I \subset(0,+\infty)$.
\end{assumption}

Under these assumptions, it is easy to see that
Assumption~\ref{hyp:concentration} is satisfied, up to an extra factor
$2$, with
\[
\psi^\star(x)= \bigl\{\mu_{\min}^2 \tilde{
\psi}^\star_{\max}(x/2) / 8 \bigr\}\wedge\psi^\star_{\bin}(x/2),
\]
where $\psi^\star_{\bin}$ is a rate function for binary
observations, defined for instance in
Equations (\ref{eq:binary_rate_function_hoeffding}) or
(\ref{eq:binary_rate_function_bernstein}). Namely, using the same
notation as in Assumption~\ref{hyp:concentration}, we get
\begin{eqnarray*}
\mathbb{P}_{\star}^{\mathbf{z}_n^{\star}\mathbf{w}_m^{\star}} \bigl( \bigl\llvert \bigl(
\delta^{\bpi} - \Delta^{\bpi}\bigr) \bigl(\mathbf{z}_n^{\star},
\mathbf{w}_m^{\star},\mathbf {z}_n,
\mathbf{w}_m\bigr) \bigr\rrvert \ge\eps\{m r_1+n
r_2 \} \bigr) \le4\exp\bigl[- \psi^\star( \eps) \{m
r_1 +n r_2 \} \bigr] .
\end{eqnarray*}

In order to ensure Assumption~\ref{hyp:Lipschitz} on
$f(\cdot;\pi,\gamma)$, we need the same hypothesis to be satisfied on
the family $\{ \tilde{f}(\cdot;\gamma) ;\gamma\in\Gamma\}$.

\begin{assumption}
\label{hyp:Lipschitz_tildef}
There exists some positive constant $\tilde L_0$ such that for any
$\boldsymbol{\gamma},\boldsymbol{\gamma}'\in\Gamma_{\Q\calL}$ and
any $(q,l),(q',l')\in\Q\times\calL$, we have
\begin{eqnarray*}
\biggl\llvert \int_{\X} \log\frac{\tilde f(x;\gamma_{ql})}{\tilde
f(x;\gamma'_{ql})} \tilde f(x;
\gamma_{q'l'}) \,\mathrm{d}x \biggr\rrvert \le \tilde L_0\bigl\|\boldsymbol{
\gamma}-\boldsymbol{\gamma}'\bigr\|_{\infty} .
\end{eqnarray*}
\end{assumption}

Note that we
provided in the previous section many examples of families for which
this assumption is satisfied. Then, the results of Section~\ref{sec:post-distr} apply.

\section{Asymptotically decreasing connections density}
\label{sec:pi_vers0}

In this section, we explore the limiting case where the numbers of
groups $Q$ and $L$ remain constant while the connections probabilities
between groups converge to 0. This framework is interesting as it
models the case
where groups sizes increase linearly with the number of row/column
objects, while the mean number of connections (i.e., non-null
observations in the data matrix) increases only sub-linearly, mimicking
for example budget constraints in terms of global consumptions. More
precisely, we will consider two different setups, the first one being
built on the binary case developed in Section \hyperref[binary]{Binary model}
 and the second
one being built on the weighted case (also called zero-inflated model)
from Section~\ref{sec:weighted}. As in the
previous sections, we assume that $m \leq n$, view $m :=m_n$ as a
sequence depending on $n$ and state the results with respect to $n \to
+\infty$. We shall furthermore assume that the probability of
connection (binary case) or the sparsity parameter (weighted case)
$\pi_{ql, n}$ depends on $n$ and writes $\pi_{ql,n}=\xi_n\pi_{ql}$ where
$(\xi_n)_{n\ge1}$ converges to zero and $\pi_{ql}$ is a positive
constant. The sequence $(\xi_n)_{n\in\N}$ controls the overall
density of the block model and acts as
a scaling factor while the parameters $(\pi_{ql})_{(q,l)\in\Q\times
\calL}$ reflect the \emph{unscaled}
connection probabilities from the different groups. This
parametrization is analogous to the one studied in \cite
{Bickel_Chen}. We shall now assume that the unscaled
connection/sparsity probabilities are well-behaved, and introduce the
new parameter
sets denoted by $\Pi_n$ and $\Pi_{\Q\calL, n}$ to account for the
dependence on the
data size (i.e., number of rows/columns).

\begin{assumption}
\label{hyp:vitesse_parametre}
The parameter sets $\Pi_n$ and $\Pi_{\Q\calL, n}$ depend on the number
of observations and we have
\begin{eqnarray*}
 \Pi& \subset& [a, 1 - a] \qquad \mbox{for some } a
\in(0,1/2),
\\
\Pi_n &:= & \xi_n \Pi= \{\xi_n \pi; \pi\in
\Pi\} ,
\\
\Pi_{\Q\calL} & \subset& \Pi^{QL},
\\
\Pi_{\Q\calL, n} & := & \xi_n \Pi_{\Q\calL} = \{
\xi_n \bpi; \bpi\in\Pi_{\Q\calL}\} ,  %
\end{eqnarray*}
where $(\xi_n)_{n \ge1}$ is a sequence of values in $(0,1]$
converging to 0 and such that
\[
\frac{\log n} {m_n \xi_n} \mathop{\rightarrow}_{n \to+\infty} 0.
\]
\end{assumption}

\subsection{Binary block models with a vanishing density}
\label{sec:pi_vers0_binaire}

We let $X_{ij} \in\{0,1\}$ and $f(\cdot;\pi)$ a Bernoulli
distribution with parameter $\pi$. Here, the connectivity parameter
$\bpi_n=(\pi_{ql, n})_{(q,l)\in\Q\times\calL}$ depends on $n$ and
may be arbitrarily close to $0$. Accordingly, the constant
$\kappa_{\min}(\bpi_n)$ defined in \eqref{eq:kappa_min} depends on $n$
and is no longer bounded away from $0$. We thus reconsider
Assumptions \ref{hyp:concentration}, \ref{hyp:Lipschitz} and the
definition of $\kappa_{\min}(\bpi_n)$ to exhibit the scaling in $n$ of
several key quantities in this setup. The proof of following lemma is
postponed to the \hyperref[sec:appendix]{Appendix}.

\begin{lemma}
\label{lem:sparse-binary-block-models}
Fix two parameters $\bpi_{n} =\xi_n \bpi$ and $\bpi'_{n} =\xi_n
\bpi'$ in the set $\Pi_{\Q\calL, n} $, where $\bpi,
\bpi' \in\Pi_{\Q\calL}$. Under
Assumption~\ref{hyp:vitesse_parametre}, we have for all $n$ and all $(q,l),
(q', l') \in\Q\times\calL$
%
\begin{eqnarray}
&\kappa_{\min, n} := \kappa_{\min}\bigl(\bpi^\star_n
\bigr)  \geq\xi_n c_{\min}\bigl(\bpi^\star\bigr) ,&
\nonumber
\\
\label{eq:lipschitz_n}&\biggl\llvert \int_{\X} \log\dfrac{f(x;\pi_{ql,n})}{f(x;\pi'_{ql,n})} f(x;
\pi_{q'l', n}) \,\mathrm{d}x \biggr\rrvert \leq\dfrac{\xi_n \|\bpi-
\bpi'\|_{\infty}}{ a} ,&
\\
\label{eq:psi_n}&\psi^\star_{ n}(x) := \psi^\star(\xi_n
x)  \geq\xi_n \breve{\psi}^\star(x) , &
\end{eqnarray}
where
\begin{eqnarray*}
c_{\min} & :=&c_{\min}\bigl(\bpi^\star\bigr)\\
& =&
\frac{1}{2} \biggl( \frac
{a}{1 - a} \biggr)^2\min \biggl\{
\frac{(\pi^\star_{ql} -
\pi^\star_{q'l'})^2}{\pi^\star_{ql}}; (q,l), \bigl(q',l'\bigr)\in\Q\times
\calL,\pi^\star_{ql}\neq\pi^\star_{q'l'}
\biggr\} > 0,
\\
\breve{\psi}^\star(x) & =& \frac{x^2 \mu_{\min}^2}{64
a\{\log(1-a) - \log a\}^2 + 32 x \{\log(1-a) - \log a \}/ 3}.
\end{eqnarray*}
\end{lemma}

\begin{cor}
\label{cor:binary-models-vers0}
Under Assumption~\ref{hyp:identifiability} on the unscaled parameter
set $\Pi_{\Q\calL}$ and Assumption~\ref{hyp:vitesse_parametre}, the
conclusions of Theorem~\ref{thm:cv_posterior} and
Corollaries \ref{cor:1} to \ref{cor:3} remain valid with the
following modifications
\begin{enumerate}[3.]
\item[1.]$c = \mu_{\min}^2 c_{\min} / 16$;
\item[2.]$L_0 = a^{-1}$;
\item[3.]$(c - 2L_0 \|\bpi- \boldsymbol{\pi}^{\star}\|_{\infty})$ is
replaced by $\xi_n
(c - 2L_0 \|\bpi- \boldsymbol{\pi}^{\star}\|_{\infty})$.
\end{enumerate}
\end{cor}

\begin{remark}
Note that Assumption~\ref{hyp:vitesse_parametre} replaces
Assumption~\ref{hyp:asymptotics} in this statement. The quantity $m_n$ is
replaced by $m_n \xi_n$ which plays the role of average number of
connections and must grow faster than $\log n$. The scaling is
consistent with results from \cite{Bickel_Chen} and
\cite{Flynn_Perry}.
\end{remark}
\begin{pf*}{Proof of Corollary \ref{cor:binary-models-vers0}}
The proof is essentially the same as the one of
Theorem~\ref{thm:cv_posterior}. We will only highlight the
differences and show how the scaling $\log n/(m_n \xi_n) \to0$ is
derived. First, Equation (\ref{eq:CondExp_Order_Inf}) from
Proposition~\ref{prop:CondExp_Order} now depends on $n$ and should
be
%
\begin{equation}
\label{eq:CondExp_Order_Inf_n} \mathbb{E}_{\star}^{\mathbf{Z}_n\mathbf{W}_{m}} \bigl(
\delta^{\bpi
_n}(\mathbf{Z}_n,\mathbf{W}_{m},
\mathbf{z}_n,\mathbf{w}_m) \bigr) \ge2\xi_n
\bigl(c' - L_0\bigl\|\bpi-\boldsymbol{\pi}^{\star}
\bigr\|_{\infty}\bigr) (mr_1 + nr_2),
\end{equation}
where the original $c = \mu_{\min}^2 \kappa_{\min} / 16$ has been
changed to $c' = \mu_{\min}^2 c_{\min} / 16$. Next, the set
$A(\mathbf{z}_n^{\star},\mathbf{w}_m^{\star},\mathbf{z}_n,\mathbf
{w}_m)$ must be changed so that we consider two-sided
deviations between $\delta^{\bpi_n}(\mathbf{Z}_n,\allowbreak \mathbf
{W}_{m},\mathbf{z}_n,\mathbf{w}_m)$ and its
conditional expectation of order $\xi_n (c' -
L_0\|\bpi-\boldsymbol{\pi}^{\star}\|_{\infty}) (mr_1 + nr_2)$
instead of the previous
$(c - L_0\|\bpi-\boldsymbol{\pi}^{\star}\|_{\infty}) (mr_1 +
nr_2)$. Equation (\ref{eq:concentration2}) therefore turns to
%
\begin{eqnarray*}
&&\mathbb{P}_{\star}\bigl(A\bigl(\mathbf{z}_n^{\star},
\mathbf{w}_m^{\star
},\mathbf{z}_n,
\mathbf{w}_m\bigr)\cap\bigl\{(\mathbf{Z}_n,
\mathbf{W}_{m}) = \bigl(\mathbf{z}_n^{\star},
\mathbf{w}_m^{\star}\bigr)\bigr\}\bigr)
\\
&&\quad \le2\exp\bigl[- \psi_n^\star\bigl(c'-2L_0
\eta\bigr) (mr_1 +nr_2 ) \bigr] \bmu\bigl(
\mathbf{z}_n^{\star},\mathbf{w}_m^{\star}
\bigr)
\\
&&\quad \le2\exp\bigl[- \breve{\psi}^\star\bigl(c'-2L_0
\eta\bigr) \xi_n (mr_1 +nr_2 ) \bigr] \bmu
\bigl(\mathbf{z}_n^{\star},\mathbf{w}_m^{\star}
\bigr) .
\end{eqnarray*}
The set $\Omega_1$ is still defined as in Equation (\ref{eq:Omega1})
and on this set, Inequality (\ref{eq:cv_post_ratio_delta}) and thus
both \eqref{eq:cv_post_ratio_1} and \eqref{eq:cv_post_ratio_2} are
still satisfied. However, the upper bound on
$\mathbb{P}_{\star}(\overline{\Omega_1})$ is modified as follows
for LBM
\begin{eqnarray*}
\mathbb{P}_{\star}(\overline\Omega_1) &\le&
\mathbb{P}_{\star
}(\overline\Omega_0)\\
&&{} + 2 |\Sig| \bigl[ \bigl
\{1+\exp\bigl[-m\xi_n \breve{\psi}^\star\bigl(c'-2L_0
\eta\bigr) \bigr]\bigr\}^{n} \bigl\{1+\exp\bigl[-n\xi_n
\breve{\psi}^\star\bigl(c'-2L_0\eta\bigr) \bigr]
\bigr\}^{m} - 1 \bigr] .
\end{eqnarray*}
Combining the latter with the control of the probability of
$\overline\Omega_0$ given in Proposition~\ref{prop:CondExp_Order},
we obtain for LBM
\begin{eqnarray*}
\eps_{n,m}:=\mathbb{P}_{\star}(\overline\Omega_1)
\le2QL \exp \bigl[-(n \wedge m) \mu_{\min}^2/2\bigr] + 2 |
\Sig| d_{n,m} \exp(d_{n,m}),
\end{eqnarray*}
where $d_{n,m} = [n\exp\{-m\xi_n \breve{\psi}^\star(c'-2L_0\eta)
\}
+ m\exp\{- n\xi_n \breve{\psi}^\star(c'-2L_0\eta) \}]$. The
condition required to make the $\eps_{n,m}$ summable and conclude
the proof is $\log n/(m\xi_n) \to0$. This condition holds under
Assumption~\ref{hyp:vitesse_parametre}. Note that for SBM, we get
\begin{eqnarray*}
\eps_{n,n}:=\mathbb{P}_{\star}(\overline\Omega_1)
\le2Q \exp\bigl[-n \alpha _{\min}^2/2\bigr] + 2 |\Sig|
d_{n} \exp(d_{n}),
\end{eqnarray*}
with $d_{n}=n\exp\{-2n\xi_n \breve{\psi}^\star(c'-2L_0\eta)\}$ and
which also satisfies $\sum_{n}\eps_{n,n}<+\infty$.
\end{pf*}

\subsection{Weighted models with a vanishing density}
\label{sec:pi_vers0_value}

We now consider the setup introduced in Section~\ref{sec:weighted} as
well as corresponding assumptions, except that we shall now assume
that the sparsity parameters $\pi_{ql, n} :=\xi_n \pi_{ql}$ may
be arbitrarily close to zero (see
Assumption~\ref{hyp:vitesse_parametre}). Note that the parameters
$(\gamma_{ql})_{(q,l)\in\Q\times\calL} \in\Gamma_{\Q\calL}$
remain fixed.
Flynn and Perry \cite{Flynn_Perry} adopt a sparse setup where the average
entry value goes to $0$. Our setup with inflated numbers of
$0$-valued entry is only a special instance of theirs but is
in our opinion a realistic way to model sparse matrices.

In the next lemma, we provide the scaling of $\kappa_{\min}(\bpi_n,
\gamma)$, or more accurately a lower bound thereof, and show that
Assumption~\ref{hyp:Lipschitz_tildef} is sufficient to
guarantee the adequate scaling of the Lipschitz condition. We
however need a stronger condition that in
Assumption~\ref{hyp:concentration_tildef} to control deviations in
this setup.

\begin{assumption}
\label{hyp:concentration_sparse_tildef}
Fix $(\mathbf{z}_n^{\star},\mathbf{w}_m^{\star}) \in\tU_0$ and $
(\mathbf{z}_n,\mathbf{w}_m)$ in $\tU$ with
$(\mathbf{z}_n^{\star},\mathbf{w}_m^{\star}) \neq(\mathbf
{z}_n,\mathbf{w}_m)$. Let $\tilde Y_{ij}=\log[\tilde
f(X_{ij} ; \gamma_{z^\star_iw^\star_j})/\tilde f
(X_{ij};\gamma_{z_iw_j})] + c$, where $c$ is a centering constant.
There exists some interval $I\subset(0, +\infty)$ such that the
function $\tilde{\psi}^\star$ defined on $(0,+\infty)$ as
%
\begin{equation}
\label{eq:concentration_sparse_tildef} \tilde{\psi}^\star(x) = \sup_{\lambda\in I }
\Bigl(\lambda x - \sup_{\boldsymbol{\gamma} \in\Gamma_{\Q\calL}} \max_{(i,j) \in\I}
\mathbb{E}_{\star}^{\mathbf{z}_n^{\star
}\mathbf{w}_m^{\star}}\bigl[\exp(\lambda\tilde{Y}_{ij}
) | X_{ij} \neq0\bigr] + 1 \Bigr)
\end{equation}
exists and is positive on $(0,+\infty)$.
\end{assumption}

\begin{remark}
Note that $\tilde{\psi}^\star(x)$ is \emph{not} the rate from the
usual Cramer--Chernoff bound, which is in general $\sup_{\lambda>
0} ( \lambda x - \log\esp[\exp(\lambda Z)] ) $ for a centered
random variable $Z$. The part $\log\esp[\exp(\lambda Z)]$ is
replaced with the larger quantity $\exp(\lambda Z) - 1$ which
induces a slower exponential decrease. The formula arises from a
Taylor expansion of rate function $\psi^\star$ for deviations of
order $\xi_n x$ to obtain a linear scaling of $\psi^\star(\xi_n
x)$ with respect to $\xi_n$. Note also that the condition that
$\tilde{\psi}^\star$ positive in
Assumption~\ref{hyp:concentration_sparse_tildef} is stronger than
and implies corresponding
Assumption~\ref{hyp:concentration_tildef}.
\end{remark}

The proof of following lemma is postponed to the \hyperref[sec:appendix]{Appendix}.

\begin{lemma}
\label{lem:sparse-weighted-models}
Fix two parameters $\bpi_{n} =\xi_n \bpi$ and $\bpi'_{n} =\xi_n
\bpi'$ in the set $\Pi_{\Q\calL, n} $, where $\bpi, \bpi' \in
\Pi_{\Q\calL}$. Under
Assumptions \ref{hyp:Lipschitz_tildef}, \ref
{hyp:concentration_sparse_tildef} and using the notation of Section~\ref{sec:weighted},
we have for all $n$, all $(q,l), (q', l') \in\Q\times\calL$
and all $\boldsymbol{\gamma, \gamma'} \in\Gamma_{\Q\calL}$,
%
\begin{eqnarray}
\label{eq:kappa_n_weighted}&\kappa_{\min, n} :=\kappa_{\min}\bigl(\xi_n
\bpi^\star, \boldsymbol{\gamma}^\star\bigr) \geq
\xi_n \bigl( c_{\min}\bigl(\bpi^\star\bigr) + a
\kappa_{\min}\bigl(\boldsymbol{\gamma}^\star\bigr) \bigr)
,&
\\
\label{eq:lipschitz_n_weighted}&\biggl\llvert \displaystyle \int_{\X} \log\dfrac{f(x;\pi_{ql,n},
\gamma_{ql})}{f(x;\pi'_{ql,n}, \gamma'_{ql})} f(x;
\pi_{q'l',
n}, \gamma_{q' l'})\, \mathrm{d}x \biggr\rrvert \leq
\xi_n \biggl( \dfrac{\|\bpi
- \bpi'\|_{\infty}}{a} + \tilde{L}_0 \bigl\|
\boldsymbol{\gamma} - \boldsymbol{\gamma}'\bigr\|_{\infty} \biggr)
, &
\\
\label{eq:psi_n_weighted}&\psi^\star_{ n}(x) :=\psi^\star(\xi_n
x)  \geq\dfrac{\xi_n
\mu_{\min}^2}{8} \biggl( \tilde{\psi}^\star\biggl(
\dfrac{x}{2}\biggr) \wedge \breve{\psi}^\star\biggl(
\dfrac{x}{2}\biggr) \biggr) ,&
\end{eqnarray}
where
\begin{eqnarray*}
\kappa_{\min} &:=&\kappa_{\min}\bigl(\boldsymbol{
\gamma}^\star\bigr)  = \min \bigl\{ \tilde D\bigl(\gamma^\star_{ql}
\parallel \gamma^\star_{q'l'}\bigr) ; (q,l), \bigl(q',l'
\bigr)\in\Q\times\calL,\gamma^\star_{ql}\neq
\gamma^\star_{q'l'} \bigr\} > 0 ,
\\
\tilde D\bigl(\gamma\parallel \gamma'\bigr) & :=&\int_{\X}
\log \biggl(\frac
{\tilde f(x;\gamma)}{\tilde f(x;\gamma')} \biggr) \tilde f(x;\gamma)\,\mathrm{d}x ,\qquad  \forall\gamma,
\gamma' \in\Gamma,
\\
c_{\min} &:=&c_{\min}\bigl(\bpi^\star\bigr)  \\
&=&
\frac{1}{2} \biggl( \frac{a}{1 - a} \biggr)^2 \min \biggl\{
\frac{(\pi^\star_{ql} -
\pi^\star_{q'l'})^2}{\pi^\star_{ql}} ; (q,l), \bigl(q',l'\bigr)\in \Q
\times\calL,\pi^\star_{ql}\neq\pi^\star_{q'l'}
\biggr\} > 0,
\\
\breve{\psi}^\star(x) & =& \frac{x^2}{8 a[\log(1-a) - \log a]^2 + 4
x [\log(1-a)
- \log a ]/ 3},
\\
\tilde{\psi}^\star(x) & =& \sup_{\lambda\in I } \Bigl(\lambda x
- \max_{(i,j) \in\I} \mathbb{E}_{\star}^{\mathbf{z}_n^{\star
}\mathbf{w}_m^{\star}}\bigl[
\exp(\lambda\tilde{Y}_{ij}) | X_{ij} \neq0\bigr] + 1 \Bigr).
\end{eqnarray*}
\end{lemma}

\begin{cor}
\label{cor:weighted-models-vers0}
Under the assumptions from Section~\ref{sec:weighted}
and Assumption~\ref{hyp:concentration_sparse_tildef} replacing
the weaker Assumption~\ref{hyp:concentration_tildef},
Theorem~\ref{thm:cv_posterior} and Corollaries \ref{cor:1} to
\ref{cor:3} remain valid with the following modifications
\begin{enumerate}[3.]
\item[1.]$L_0 = a^{-1} + \tilde{L_0}$;
\item[2.]$c = \mu_{\min}^2 (c_{\min} + a \kappa_{\min})/16$;
\item[3.]$\bpi$ is replaced by $(\xi_n \bpi, \boldsymbol{\gamma})$;
\item[4.]$(c - 2L_0 \|\bpi- \boldsymbol{\pi}^{\star}\|_{\infty})$ is
replaced by $\xi_n
(c - 2L_0 \| (\bpi, \boldsymbol{\gamma}) - (\boldsymbol{\pi
}^{\star},
\boldsymbol{\gamma}^\star) \|_{\infty})$.
\end{enumerate}
\end{cor}

\begin{pf}
This result is proved following the proof of
Theorem~\ref{thm:cv_posterior}, exactly in the same way as we did
for Corollary~\ref{cor:binary-models-vers0}, with some changes in
key quantities as listed in the corollary.
\end{pf}


\begin{appendix}

\section*{Appendix: Technical proofs}
\label{sec:appendix}
\renewcommand{\theequation}{\arabic{equation}}
\setcounter{equation}{39}
\begin{pf*}{Proof of Lemma~\ref{lem:Bound_Number}}
Let us recall that this proof is a generalization of the proof of
Proposition B.5 in \cite{Alain_JJ_LP}.

Since $(\mathbf{z}_n^{\star},\mathbf{w}_m^{\star})\in\tU^0$, for
any $q\in\Q$ and any $l\in\calL
$, the number of entries in
$\mathbf{z}_n^{\star}$ (resp. in $\mathbf{w}_m^{\star}$) which take
value $q$ (resp. $l$) is at least
$\lceil
n\mu_{\min}/2\rceil$ (resp. $\lceil m\mu_{\min}/2\rceil$). Up to
a reordering of the vectors $\mathbf{z}_n^{\star}$ and $\mathbf
{w}_m^{\star}$, we may assume that the
first $Q\lceil n\mu_{\min}/2\rceil$ entries of $\mathbf{z}_n^{\star
}$ and the first
$L\lceil m\mu_{\min}/2\rceil$ entries of $\mathbf{w}_m^{\star}$
are fixed, with
%
\begin{eqnarray}\label{eq:order_zw}
\mathbf{z}_n^{\star}&=&\bigl(1,2,\ldots,Q,1,2,\ldots, Q,
\ldots, 1,2,\ldots, Q, z^\star _{Q\lceil
n\mu_{\min}/2\rceil+1}, \ldots,
z^\star_n\bigr) ,
\nonumber
\\[-8pt]\\[-8pt]
\mathbf{w}_m^{\star}&=&\bigl(1,2,\ldots,L,1,2,\ldots, L,
\ldots, 1,2,\ldots, L, w^\star _{L\lceil
m\mu_{\min}/2\rceil+1}, \ldots,
w^\star_m\bigr) .\nonumber
\end{eqnarray}
Such ordering of the entries of $(\mathbf{z}_n^{\star},\mathbf
{w}_m^{\star})$ induces a specific
ordering of the entries of $(\mathbf{z}_n,\mathbf{w}_m)$. For each
$k\in
\{1,\ldots, \lceil n\mu_{\min}/2\rceil\}$ (resp. each $j\in
\{1,\ldots, \lceil m\mu_{\min}/2\rceil\}$), we denote by $s_k$
(resp. $t_j$) the application from $\Q$ to $\Q$ (resp. from $\calL$
to $\calL$) defined by
\[
\forall q\in\Q,\qquad  s_k\bigl( z^\star_{(k-1)Q+q}\bigr) =
z_{(k-1)Q+q} \quad \mbox{and}\quad  \forall l\in\calL,\qquad  t_j\bigl(
w^\star_{(j-1)L+l}\bigr) = w_{(j-1)L+l} .
\]
In other words, we write $\mathbf{z}_n$ and $\mathbf{w}_m$ in the form
%
\begin{eqnarray}\label{eq:order_z'w'}
\mathbf{z}_n&=&\bigl(s_1(1),s_1(2),
\ldots,s_1(Q),s_2(1),\ldots, s_2(Q), \ldots,
s_{\lceil n\mu_{\min}/2\rceil}(1),\ldots, s_{\lceil
n\mu_{\min}/2\rceil}(Q),
\nonumber
\\
&&\hphantom{\bigl(} z_{Q\lceil n\mu_{\min}/2\rceil+1}, \ldots, z_n\bigr)
\nonumber
\\[-8pt]\\[-8pt]
\mathbf{w}_m&=&\bigl(t_1(1),t_1(2),
\ldots,t_1(L),t_2(1),\ldots, t_2(L), \ldots,
t_{\lceil m\mu_{\min}/2\rceil}(1),\ldots, t_{\lceil
m\mu_{\min}/2\rceil}(L),
\nonumber
\\
&& \hphantom{\bigl(}w_{L\lceil m\mu_{\min}/2\rceil+1}, \ldots, w_m\bigr) . \nonumber
\end{eqnarray}
There are several possible orderings of $\mathbf{z}_n^{\star}$ (resp.
$\mathbf{w}_m^{\star}$) in the
form \eqref{eq:order_zw} and each one
induces a different ordering of $\mathbf{z}_n$ (resp. $\mathbf{w}_m$)
in the form
\eqref{eq:order_z'w'}. For example,
for any $1 \le k, k'\le\lceil n\mu_{\min}/2 \rceil$ and any $q\in
\Q$, we can exchange $z^\star_{(k-1)Q+q}$ and $z^\star_{(k'-1)Q+q}$
which are both
equal to $q$ and this induces a permutation between $s_k(q)$ and
$s_{k'}(q)$ in $\mathbf{z}_n$.
(Similarly for any $1 \le j, j'\le\lceil m\mu_{\min}/2 \rceil$ and
any $l\in
\calL$, we can exchange $t_j(l)$ and $t_{j'}(l)$ in $\mathbf{w}_m$.)
Also, for any $i > Q \lceil
n\mu_{\min}/2 \rceil$, $z^\star_i$ is equal to some $q\in\Q$ and
can be
exchanged with $z^\star_{(k-1)Q+q}$ for any $1 \le k\le\lceil n\mu
_{\min}/2
\rceil$. This induces a permutation between $s_k(z^\star_i)$ and
$z_i$ in
$\mathbf{z}_n$. (Similarly, we can exchange $t_j(w^\star_i)$ and
$w_i$ in $\mathbf{w}_m$
for any $i > L \lceil m\mu_{\min}/2 \rceil$ and any $1 \le j\le
\lceil m\mu_{\min}/2
\rceil$.) Note also that the orderings of $\mathbf{z}_n^{\star}$ and
$\mathbf{w}_m^{\star}$ are independent.
As already said, each $s_k$ (resp. $t_j$) is a function from $\Q$ to
$\Q$
(resp. from $\calL$ to $\calL$). We can therefore choose
orderings of $\mathbf{z}_n^{\star}$ and $\mathbf{w}_m^{\star}$
which minimize the number (ranging from
$0$ to $\lceil n\mu_{\min}/2 \rceil$) of injective functions $s$
as well as the number (ranging from $0$ to $\lceil m\mu_{\min}/2
\rceil$) of injective functions $t$.

For $1 \le k \le\lceil n\mu_{\min}/2
\rceil$ and $1 \le j \le\lceil m\mu_{\min}/2 \rceil$, let
\begin{eqnarray*}
B_{kj} = \bigl|\bigl\{(q,l) \in\Q\times\calL; \pi^\star_{ql}
\neq \pi^\star_{s_k(q)t_j(l)} \bigr\}\bigr| .
\end{eqnarray*}
We have of course $\diff(\mathbf{z}_n,\mathbf{w}_m,\mathbf
{z}_n^{\star},\mathbf{w}_m^{\star}) \ge
\sum_{k=1}^{\lceil n\mu_{\min}/2 \rceil} \sum_{j = 1}^{\lceil
m\mu_{\min}/2 \rceil} B_{kj}$.

The simplest case is obtained when for any $(k,j)$, we have
$B_{k,j}\ge1$ and then
\[
\diff\bigl(\mathbf{z}_n,\mathbf{w}_m,
\mathbf{z}_n^{\star},\mathbf {w}_m^{\star}
\bigr) \ge \biggl\lceil\frac{n\mu_{\min}}{2} \biggr\rceil\times \biggl\lceil
\frac{
m\mu_{\min}}{2} \biggr\rceil\ge\frac{\mu_{\min}^2}{8} (mr_1 +
nr_2),
\]
since both $r_1\le n$ and $r_2\le m$. In this case, the proof is finished.

Otherwise, there is at least one $(k,j)$ such that $B_{kj}=0$. In this
case, we start by proving that at least one application among the
$s_{k'}$ and
at least one application among the $t_{j'}$ are permutations.
Indeed, consider some $(k,j)$ with $B_{kj} = 0$. Assume that $s_k(q) =
s_k(q')$ for some $q \neq q'$. Then for all $l$, we have $\pi^\star
_{ql} = \pi^\star_{s_k(q)t_j(l)} = \pi^\star_{s_k(q')t_j(l)} = \pi
^\star_{q'l}$
which contradicts Assumption~\ref{hyp:identifiability}. The same
holds if $t_j(l) = t_j(l')$ for some $l \neq l'$. Therefore if $B_{kj}
= 0$, both $s_k$ and $t_j$ are injections and therefore permutations.

Now, we prove that all applications $s_{k'}$ which are
permutations are in fact equal.
Indeed, consider $k' \neq k$ such that $s_{k'}$ and $s_k$ are injections.
Assume there exists some $q$ such that $s_k(q) \neq s_{k'}(q)$. Then
exchanging $s_k(q)$ and $s_{k'}(q)$ in $\mathbf{z}_n$ decreases the
number of
injective applications $s_i$ by $2$, in contradiction with the
minimality of the
chosen ordering of coordinates in $\mathbf{z}_n^{\star}$. Therefore,
$s_k = s_{k'}$.
Thus, all injective $s_{k'}$ are equal to the same permutation $s\in
\Sig_Q$.
Similarly, all injective $t_{j'}$ are equal to the same permutation
$t\in\Sig_L$.
Since one of these pairs of permutations $(s_k,t_j)$ is associated to
the event $B_{kj}=0$, this implies that $(\boldsymbol{\pi}^{\star
})^{s,t}=\bpi^\star$.
Note also that according to Assumption~\ref{hyp:symetrie}, we
necessarily have $(s,t)\in\Sig$.

We now argue that as soon as there is at least one injective
application $s_k$ (which is thus equal to $s$), we must have $z_{i} =
s(z^\star_{i})$ for all $i \ge Q\lceil n\mu_{\min}/2\rceil+1$.
Otherwise, we could decrease
by one the total number of injective $s_{k'}$ by permuting $z_i$ and
$s(z^\star_i)$, which contradicts the minimality of the number of
injections. In the same way, if there is at least one injective
application $t_j$ (thus equal to $t$), we have $w_{i} = t(w^\star
_{i})$ for any $i \ge
L\lceil m\mu_{\min}/2\rceil+1$.

Let $d_1$ (resp. $d_2$) be the number (possibly equal to $0$) of
non-injective $s_k$
(resp. $t_j$). It comes from the two previous points that we can in
fact write
\begin{eqnarray*}
\mathbf{z}_n&=&\bigl(s_1(1),\ldots,s_1(Q),
\ldots, s_{d_1}(1),\ldots, s_{d_1}(Q), s\bigl(z^\star_{d_1Q+1}
\bigr),\ldots,s\bigl(z^\star _n\bigr)\bigr),
\\
\mathbf{w}_m&=&\bigl(t_1(1),\ldots,t_1(L),
\ldots, t_{d_2}(1),\ldots, t_{d_2}(L), t\bigl(w^\star_{d_2L+1}
\bigr),\ldots,t\bigl(w^\star_m\bigr)\bigr),
\end{eqnarray*}
where $(s,t)\in\Sig$.
Thus, we obtain that
\begin{eqnarray*}
r_1&=&d\bigl(\mathbf{z}_n,\mathbf{z}_n^{\star}
\bigr)\le\bigl\|\mathbf {z}_n-s\bigl(\mathbf{z}_n^{\star}
\bigr)\bigr\|_0\le d_1Q,
\\
r_2&=&d\bigl(\mathbf{w}_m,\mathbf{w}_m^{\star}
\bigr)\le\bigl\|\mathbf {w}_m-t\bigl(\mathbf{w}_m^{\star}
\bigr)\bigr\|_0\le d_2L.
\end{eqnarray*}
Finally, for each $(k,j)$ such that either $s_k$ or $t_j$
is non-injective, we have $B_{kj} \ge1$. Therefore,
\begin{eqnarray*}
\diff\bigl(\mathbf{z}_n,\mathbf{w}_m,
\mathbf{z}_n^{\star},\mathbf {w}_m^{\star}
\bigr) & \ge& \sum_{k=1}^{\lceil n\mu_{\min}/2
\rceil} \sum
_{j = 1}^{\lceil m\mu_{\min}/2 \rceil} B_{kj}
\\
& \ge& d_1 \lceil m\mu_{\min}/2 \rceil+ d_2
\lceil n\mu_{\min}/2 \rceil- d_1 d_2
\\
& \ge& \frac{d_1 \lceil m\mu_{\min}/2
\rceil+ d_2\lceil n\mu_{\min}/2 \rceil}{2}
\\
& \ge& \frac{r_1 \lceil m\mu_{\min}/2\rceil+r_2 \lceil n\mu_{\min
}/2\rceil}{2Q}
\\
& \ge& \frac{\mu_{\min}^2}{4}( mr_1 +nr_2 ) ,
\end{eqnarray*}
where the last inequality comes from $\mu_{\min} \le1/Q$. This
concludes the proof of the lemma.
\end{pf*}
\begin{pf*}{Proof of Lemma~\ref{lem:sparse-binary-block-models}}
For any $\pi, \pi' \in\Pi$ and any $\xi\in(0,1)$, the
Kullback--Leibler divergence $D(\xi\pi\parallel \xi\pi')$ writes
\begin{eqnarray*}
D\bigl(\xi\pi\parallel \xi\pi'\bigr) & = & \xi\pi\log\frac{\pi}{\pi'} +
(1 - \xi \pi) \log \biggl( \frac{1 - \xi\pi}{1 - \xi\pi'} \biggr)
\\
& = & -\xi\pi\log \biggl( 1 + \frac{\pi' - \pi}{\pi} \biggr) -(1 - \xi\pi) \log
\biggl( 1 + \frac{\xi(\pi- \pi')}{1 - \xi\pi} \biggr).
\end{eqnarray*}
Now, relying on the convexity inequality $\log(1+x) \leq x$ valid
for $x > -1$ and on a Taylor series expansion of $\log(1+x)$, there
exists some $\theta$ with $|\theta|\le|\pi' -\pi|/\pi$ such that
\begin{eqnarray*}
D\bigl(\xi\pi\parallel \xi\pi'\bigr) & \geq& \xi\bigl(\pi-
\pi'\bigr) + \xi\frac{(\pi- \pi')^2}{2\pi} \frac{1}{(1 + \theta)^2} - \xi\bigl(\pi-
\pi'\bigr)
\\
& \geq& \xi\frac{(\pi- \pi')^2}{2\pi} \biggl( \frac{a}{1 -
a} \biggr)^2 .
\end{eqnarray*}
Coming back to the definition \eqref{eq:kappa_min} of
$\kappa_{\min}(\bpi^\star_n)$ yields
\begin{eqnarray*}
\kappa_{\min, n} & : = & \kappa_{\min}\bigl(\bpi^\star_n
\bigr) = \kappa_{\min}\bigl(\xi_n \bpi^\star\bigr)
\\
& =& \min \bigl\{D\bigl(\xi_n \pi^\star_{ql}\parallel
\xi_n \pi^\star_{q'l'}\bigr) ; (q,l),
\bigl(q',l'\bigr)\in\Q\times\calL,\pi^\star_{ql}
\neq\pi^\star_{q'l'}\bigr\}
\\
& \geq& \xi_n c_{\min}\bigl(\bpi^\star\bigr) ,
\qquad \mbox{for all } n.
\end{eqnarray*}
Note that $\kappa_{\min, n}$ scales with $\xi_n$ only when $\Pi$ is
bounded away from $0$ and $1$. Otherwise a simple bound based on the
comparison between Kullback--Leibler divergence and the total
variation metric shows that $\kappa_{\min, n}$ scales with
$\xi_n^2$.

A similar scaling can be found to replace
Assumption~\ref{hyp:Lipschitz}. Indeed, for any $\pi, \pi', \pi''
\in\Pi$ and $\xi> 0$, we have in the binary case
\begin{eqnarray*}
\biggl\llvert \int_{\X} \log\frac{f(x;\xi\pi)}{f(x;\xi\pi')} f\bigl(x;\xi
\pi''\bigr) \,\mathrm{d}x \biggr\rrvert = \biggl\llvert \xi
\pi'' \log\frac{\pi}{\pi
'} + \bigl(1 - \xi
\pi''\bigr) \log \biggl( \frac{1 - \xi\pi}{1 - \xi\pi'} \biggr)
\biggr\rrvert \leq\frac{\xi|\pi- \pi'|}{a} .
\end{eqnarray*}
Therefore, for any $(q, l), (q', l') \in\Q\times\calL$,
\begin{eqnarray*}
\biggl\llvert \int_{\X} \log\frac{f(x;\pi_{ql,n})}{f(x;\pi'_{ql,n})} f(x;
\pi_{q'l', n}) \,\mathrm{d}x \biggr\rrvert \leq\frac{\xi_n \|\bpi-
\bpi'\|_{\infty}}{a} .
\end{eqnarray*}

Finally, the scaling for $\psi^\star_n(x)$ in
Equation (\ref{eq:psi_n}) results from Bernstein's inequality as
in~\eqref{eq:binary_rate_function_bernstein}.
\end{pf*}
\begin{pf*}{Proof of Lemma~\ref{lem:sparse-weighted-models}}
For all $\pi, \pi', \pi'' \in\Pi$, $\gamma, \gamma', \gamma''
\in
\Gamma$ and $\xi> 0$, we have
%
\begin{eqnarray}
\label{eq:general_n_weighted} \int_{\X} \log\frac{f(x;\xi\pi,
\gamma)}{f(x;\xi\pi', \gamma')} f\bigl(x;
\xi\pi'', \gamma''\bigr)\, \mathrm{d}x& =&
\xi\pi'' \log\frac{\pi}{\pi'} + \bigl(1- \xi
\pi''\bigr) \log\frac{1 - \pi}{1 - \pi'}\nonumber
\\[-8pt]\\[-8pt]
&&{}+ \xi\pi'' \int_{\X} \log
\frac{\tilde{f}(x; \gamma)}{\tilde{f}(x;\gamma')} \tilde{f}\bigl(x;\gamma''\bigr) \,\mathrm{d}x.\nonumber
\end{eqnarray}
When $(\pi'', \gamma'') = (\pi, \gamma)$,
Equation (\ref{eq:general_n_weighted}) turns to
\begin{eqnarray*}
D \bigl( (\xi\pi, \gamma) \parallel \bigl(\xi\pi', \gamma'
\bigr) \bigr) &=& D\bigl(\xi\pi\parallel \xi\pi'\bigr) + \xi\pi\tilde D\bigl(
\gamma\parallel \gamma'\bigr) \\
&\geq&\xi\frac{(\pi- \pi')^2}{2\pi} \biggl(
\frac{a}{1 -
a} \biggr)^2 + \xi a \tilde D\bigl(\gamma\parallel
\gamma'\bigr) ,
\end{eqnarray*}
from which we can deduce Inequality (\ref{eq:kappa_n_weighted}).

For general $(\pi'', \gamma'')$,
Equation (\ref{eq:general_n_weighted}) combined with
Inequality (\ref{eq:lipschitz_n}) (that applies on the Dirac part
of the distribution) and
Assumption~\ref{hyp:Lipschitz_tildef} gives
\begin{eqnarray*}
\biggl\llvert \int_{\X} \log\frac{f(x;\xi\pi,
\gamma)}{f(x;\xi\pi', \gamma')} f\bigl(x;
\xi\pi'', \gamma''\bigr) \,\mathrm{d}x
\biggr\rrvert \leq\xi\frac{|\pi-
\pi'|}{a} + \xi\tilde{L}_0 \bigl|\gamma-
\gamma'\bigr| ,
\end{eqnarray*}
from which we can deduce Inequality (\ref{eq:lipschitz_n_weighted}).

Finally, in this setup, $\breve{\psi}^\star$ arises by using a
Bernstein's inequality instead of a Hoeffding's one to control
deviations of the Dirac part in the distribution,
see \eqref{eq:binary_rate_function_bernstein}. Function $\tilde{\psi
}^\star$
arises from the control of the deviations of the weighted part of
the distribution. Indeed, recall that
\[
\tilde Y_{ij} = \log \biggl(\frac{\tilde f(X_{ij} ;
\gamma_{z^\star_iw^\star_j})}{\tilde f
(X_{ij};\gamma_{z_iw_j})} \biggr) + c,
\]
where $c$ is a centering constant and let
\[
\bar Y_{ij} :=1\{X_{ij}\neq0\} \tilde Y_{ij}.
\]
We get for $\lambda\in I$,
\begin{eqnarray*}
\mathbb{E}_{\star}^{\mathbf{z}_n^{\star}\mathbf{w}_m^{\star}}\bigl[ \mathrm{e}^{ \lambda\bar{Y}_{ij}}\bigr] & =&
\mathbb{E}_{\star}^{\mathbf{z}_n^{\star}\mathbf{w}_m^{\star
}} \bigl[ 1\{X_{ij}\neq0\}
\mathrm{e}^{\lambda
\tilde{Y}_{ij}} + \bigl(1 - 1\{X_{ij}\neq0\}\bigr) \bigr]
\\
& =& 1 + \mathbb{E}_{\star}^{\mathbf{z}_n^{\star}\mathbf
{w}_m^{\star}}\bigl[ 1\{X_{ij}\neq0
\} \bigl(\mathrm{e}^{\lambda
\tilde{Y}_{ij}} - 1\bigr)\bigr]
\\
& =& 1 + \xi_n \pi^\star_{z_i^\star w_j^\star} \bigl[ \mathbb
{E}_{\star}^{\mathbf{z}_n^{\star}
\mathbf{w}_m^{\star}} \bigl(\mathrm{e}^{\lambda\tilde{Y}_{ij}} | X_{i,j}\neq0
\bigr) - 1 \bigr],
\end{eqnarray*}
from which we deduce
\begin{eqnarray*}
\mathbb{P}_{\star}^{\mathbf{z}_n^{\star}\mathbf{w}_m^{\star}} \bigl( |\bar Y_{ij} | \ge
\xi_n x \bigr) & \leq& \inf_{\lambda\in I } \exp \bigl[ \log\bigl(
\mathbb{E}_{\star}^{\mathbf
{z}_n^{\star}\mathbf{w}_m^{\star}} \bigl[ \mathrm{e}^{ \lambda\bar{Y}_{ij}} \bigr] - \lambda
\xi_n x\bigr) \bigr]
\\
& =& \inf_{\lambda\in I } \exp \bigl[ \log \bigl( 1 + \xi_n
\pi^\star_{z_i^\star w_j^\star} \bigl( \mathbb{E}_{\star
}^{\mathbf{z}_n^{\star}
\mathbf{w}_m^{\star}}
\bigl[ \mathrm{e}^{\lambda\tilde{Y}_{ij}} | X_{i,j}\neq0\bigr] - 1 \bigr) \bigr) - \lambda
\xi_n x \bigr]
\\
& \leq&\exp \Bigl[ \inf_{\lambda\in I } \xi_n \bigl(
\pi^\star_{z_i^\star w_j^\star} \bigl( \mathbb{E}_{\star}^{\mathbf
{z}_n^{\star}
\mathbf{w}_m^{\star}}
\bigl[ \mathrm{e}^{\lambda\tilde{Y}_{ij}} | X_{i,j}\neq0\bigr] - 1 \bigr) - \lambda x
\bigr) \Bigr]
\\
& \leq&\exp \Bigl[ \xi_n \sup_{\lambda\in I} \Bigl( \lambda
x + 1 - \max_{(i,,j) \in\I} \mathbb{E}_{\star}^{\mathbf{z}_n^{\star
}\mathbf{w}_m^{\star}}
\bigl[ \mathrm{e}^{\lambda
\tilde{Y}_{ij}} | X_{i,j}\neq0\bigr] \Bigr) \Bigr],
\end{eqnarray*}
where we used that $\log(1+x) \le x$. The last inequality gives us
the rate function
$\tilde{\psi}^\star$. Equation (\ref{eq:psi_n_weighted}) is
constructed from $\tilde{\psi}^\star$ and $\breve{\psi}^\star$ as
in Section~\ref{sec:weighted}.
\end{pf*}

\end{appendix}




\printhistory

\end{document}